\documentclass[12pt,reqno]{article}%
\usepackage{amsmath,amsthm,amsfonts,amssymb,amscd}
\usepackage{graphicx}
\usepackage{verbatim}
\usepackage{tikz}
\usepackage{makeidx}
\usepackage{indentfirst}
\usepackage{multicol}
\usepackage{verbatim}
\usepackage{ifthen}
\usepackage{multicol}
\usepackage{amsmath}
\usepackage{amsfonts}
\usepackage{amssymb}
\usepackage{verbatim}%
\newtheorem{thm}{Theorem}[section]
\newtheorem{cor}[thm]{Corollary}
\newtheorem{lem}[thm]{Lemma}
\newtheorem{prop}[thm]{Proposition}
\newtheorem{exam}[thm]{Example}
\theoremstyle{definition}
\newtheorem{defn}[thm]{Definition}
\theoremstyle{remark}
\newtheorem{rem}[thm]{Remark}
\numberwithin{equation}{section}

\usetikzlibrary{matrix}
\begin{document}

\title{On the Structure of Irreducible Yetter-Drinfeld Modules over Quasi-Triangular
Hopf Algebras}
\author{Zhimin Liu\thanks{E-mail: zhiminliu13@fudan.edu.cn} \hspace{1cm} Shenglin Zhu
\thanks{\textbf{CONTACT}: mazhusl@fudan.edu.cn, School of Mathematical
Sciences, Fudan University, Shanghai 200433, China.} \thanks{This work was
supported by NNSF of China (No. 11331006).}\\Fudan University, Shanghai 200433, China}
\date{}
\maketitle

\begin{abstract}
Let $\left(  H,R\right)  $ be a finite dimensional semisimple and cosemisimple
quasi-triangular Hopf algebra over a field $k$. In this paper, we give the
structure of irreducible objects of the Yetter-Drinfeld module category ${}%
{}_{H}^{H}\mathcal{YD}.$ Let $H_{R}$ be the Majid's transmuted braided group
of $\left(  H,R\right)  ,$ we show that $H_{R}$ is cosemisimple. As a
coalgebra, let $H_{R}=D_{1}\oplus\cdots\oplus D_{r}$ be the sum of minimal
$H$-adjoint-stable subcoalgebras. For each $i$ $\left(  1\leq i\leq r\right)
$, we choose a minimal left coideal $W_{i}$ of $D_{i}$, and we can define the
$R$-adjoint-stable algebra $N_{W_{i}}$ of $W_{i}$. Using Ostrik's theorem on
characterizing module categories over monoidal categories, we prove that
$V\in{}_{H}^{H}\mathcal{YD}$ is irreducible if and only if there exists an $i$
$\left(  1\leq i\leq r\right)  $ and an irreducible right $N_{W_{i}}$-module
$U_{i}$, such that $V\cong U_{i}\otimes_{N_{W_{i}}}\left(  H\otimes
W_{i}\right)  $.

Our structure theorem generalizes the results of Dijkgraaf-Pasquier-Roche and
Gould on Yetter-Drinfeld modules over finite group algebras. If $k$ is an
algebraically closed field of characteristic, we stress that the
$R$-adjoint-stable algebra $N_{W_{i}}$ is an algebra over which the dimension
of each irreducible right module divides its dimension.

\end{abstract}


\textbf{2000 MATHEMATICS SUBJECT CLASSIFICATION}: 16W30



\section{Introduction}

\label{section-intro}Yetter-Drinfeld modules over a bialgebra were introduced
by Yetter \cite{yetter1990quantum} in 1990. For any finite dimensional Hopf
algebra $H$ over a field $k$, Majid \cite{majid1991doubles} identified the
Yetter-Drinfeld modules with the modules over the Drinfeld double $D(H^{cop})$
by giving the category equivalences ${}_{H}^{H}\mathcal{YD}\approx{}_{H^{cop}%
}\mathcal{YD}^{H^{cop}}\approx{}_{D(H^{cop})}\mathcal{M}$.

If $H=kG$ is the group algebra of a finite group $G$ and $k$ is algebraically
closed of characteristic zero, then the irreducible modules over $D\left(
H\right)  $ are completely described by Dijkgraaf-Pasquier-Roche in
\cite{dijkgraaf1992quasi} and independently by Gould in
\cite{gould1993quantum}. For any $g\in G$, let $C\left(  g\right)  =\left\{
h\in G\mid hgh^{-1}=g\right\}  $ be the centralizer subgroup of $g$ in $G$. If
$U$ is an irreducible module of $kC\left(  g\right)  $, then $H\otimes
_{kC\left(  g\right)  }U$ is an irreducible module in ${}_{H}^{H}\mathcal{YD}%
$, where $f\left(  h\otimes u\right)  =fh\otimes u$ and $\rho\left(  h\otimes
u\right)  =hgh^{-1}\otimes h\otimes u$ ($\forall f,h\in G$, $u\in U$). All the
irreducible modules in ${}_{H}^{H}\mathcal{YD}$ can be constructed in this
way. For detail, one can refer to \cite{dijkgraaf1992quasi}.

A natural question arises. Can we construct all the irreducible
Yetter-Drinfeld modules in ${}_{H}^{H}\mathcal{YD}$ for a more general Hopf
algebra $H$?

It might be a difficult problem to classify all the Yetter-Drinfeld modules
over general Hopf algebras. For factorizable Hopf algebras, this question was
studied by Reshetikhin and Semenov-Tian-Shansky \cite{reshetikhin1988quantum}
in 1988, where they proved that the Drinfeld double of any finite dimensional
Hopf algebra is factorizable; and if a finite dimensional Hopf algebra $H$ is
factorizable, then as Hopf algebras the Drinfeld double $D\left(  H\right)  $
is isomorphic to a twist of $H\otimes H$. Therefore, the Yetter-Drinfeld
modules can be constructed consequently. In 2001, Schneider
\cite{schneider2001some} proved that $H$ is factorizable if and only if the
double $D\left(  H\right)  $ is isomorphic to a twist of the tensor product
$H\otimes H.$

In this paper, we study the structure of Yetter-Drinfeld modules of a
semisimple and cosemisimple quasi-triangular Hopf algebra $H$ over a field. We
will give a characterization of all the irreducible modules in $_{H}%
^{H}\mathcal{YD}$.

Let $k$ be a field, and let $\left(  H,R\right)  $ be a finite dimensional
semisimple and cosemisimple quasi-triangular Hopf algebra over $k.$ Denote by
$H_{R}$ the Majid's \cite{Majid1991Braided} transmuted braided group, then
$H_{R}$ is a coalgebra on $H$ with a twisted coproduct. We show in Section
\ref{Section-struct-HR} that any Yetter-Drinfeld submodule of $H\in{}_{H}%
^{H}\mathcal{YD}$ is a subcoalgebra of $H_{R}$, and $H_{R}$ is cosemisimple.
As a Yetter-Drinfeld module, $H\in$ ${}_{H}^{H}\mathcal{YD}$ has a
decomposition
\[
H=D_{1}\oplus\cdots\oplus D_{r}\text{,}%
\]
of irreducible modules, which is also the decomposition of minimal
$H$-adjoint-stable subcoalgebras of $H_{R}$.

By a result of H. Zhu~\cite{Zhu2015Relative}, there exists a natural category
isomorphism $_{H}^{H}\mathcal{YD}\rightarrow{}_{H}^{H_{R}}\mathcal{M}$. Then,
canonically $_{H}^{H}\mathcal{YD}={}_{H}^{H_{R}}\mathcal{M}$ can be viewed as
a module category over the tensor category ${}_{H^{cop}}\mathcal{M}$. The
general theory of module categories over a monoidal category was developed by
Ostrik in \cite{Ostrik2003module}. Our decomposition
\[
H_{R}=D_{1}\oplus\cdots\oplus D_{r}%
\]
of minimal $H$-adjoint-stable subcoalgebras of $H_{R}$ yields a decomposition
\[
_{H}^{H_{R}}\mathcal{M}=\oplus{}_{i=1}^{r}{}_{H}^{D_{i}}\mathcal{M}%
\]
of indecomposable module categories $_{H}^{D_{i}}\mathcal{M}$ over the tensor
category {}$_{H^{cop}}\mathcal{M}$.

Let us recall a theorem of Ostrik~\cite[Theorem 3.1]{Ostrik2003module}. Let
$\mathcal{C}$ be a semisimple rigid monoidal category with finitely many
irreducible objects and an irreducible unit object. Let $\mathcal{M}$ be a
$k$-linear abelian semisimple module category over $\mathcal{C}$. If $M$ is a
generator of $\mathcal{M}$, then $A=\underline{\operatorname{Hom}}\left(
M,M\right)  $ is a semisimple algebra in {}$\mathcal{C}$, and the functor
$F=\underline{\operatorname{Hom}}\left(  M,\bullet\right)  :\mathcal{M}%
\rightarrow\mathrm{Mod}_{\mathcal{C}}(A)$ is a module category equivalence. If
$\mathcal{M}$ is indecomposable, then every nonzero object generates
$\mathcal{M}$.

If $\left(  X,q\right)  $ be a right $A$-module in $\mathcal{C}$, and $\left(
M,ev_{M,M}\right)  $ be the natural left $A$-module in $\mathcal{M}$, we can
define a similar tensor product $X\otimes_{A}M$ as in \cite{Ostrik2003module}.
Then we prove in Theorem \ref{theorem-G-is-qinv} that
\[
G=\bullet\otimes_{A}M:\mathrm{Mod}_{\mathcal{C}}(A)\rightarrow\mathcal{M}%
\]
is a quasi-inverse of $F=\underline{\operatorname{Hom}}\left(  M,\bullet
\right)  $. Hence $V$ is an irreducible object of $\mathcal{M}$, if and only
if there exists an irreducible object $U$ in $\mathrm{Mod}_{\mathcal{C}}(A)$,
such that $V\cong U\otimes_{A}M$.

In Section \ref{section-YDmodule}, we study the structure of irreducible
objects of%
\[
_{H}^{H}\mathcal{YD}={}_{H}^{H_{R}}\mathcal{M}=\oplus{}_{i=1}^{r}{}_{H}%
^{D_{i}}\mathcal{M}.
\]
We have two ways to view $_{H}^{D_{i}}\mathcal{M}$ as a module category over
{}$_{H^{cop}}\mathcal{M}$, or as a module category over {}$\mathrm{Vec}_{k}$.

Take for each $1\leq i\leq r$ a finite dimensional nonzero left $D_{i}%
$-comodule $W_{i}$.

Let $\mathcal{C}={}_{H^{cop}}\mathcal{M}$. Then $M_{i}=H\otimes W_{i}$ is a
generator of the indecomposable module category $_{H}^{D_{i}}\mathcal{M}$ over
{}$_{H^{cop}}\mathcal{M}$. Then Ostrik's theorem and our
Theorem~\ref{theorem-G-is-qinv} applies. We have%
\[
F_{i}=\underline{\operatorname{Hom}}\left(  M_{i},\bullet\right)  :{}%
_{H}^{D_{i}}\mathcal{M}\rightarrow\mathrm{Mod}_{\mathcal{C}}(A_{i}),%
\]
where $A_{i}=\underline{\operatorname{Hom}}\left(  M_{i},M_{i}\right)
\cong\operatorname{Hom}^{D}\left(  H\otimes W_{i},H\otimes W_{i}\right)  $, is
a category equivalence with quasi-inverse
\[
G_{i}=\bullet\otimes_{A_{i}}M_{i}:\mathrm{Mod}_{\mathcal{C}}(A_{i}%
)\rightarrow{}_{H}^{D_{i}}\mathcal{M}\text{.}%
\]

Therefore an object $V\in{}_{H}^{H}\mathcal{YD}={}_{H}^{H_{R}}\mathcal{M}%
=\oplus{}_{i=1}^{r}{}_{H}^{D_{i}}\mathcal{M}$ is irreducible, if and only if
it belongs to one of the $_{H}^{D_{i}}\mathcal{M}$ for certain $1\leq i\leq
r$, and there exists an irreducible object $U$ of $\mathrm{Mod}_{\mathcal{C}%
}(A_{i})$, such that $U\otimes_{A_{i}}M_{i}\cong V$ as objects in {}%
$_{H}^{D_{i}}\mathcal{M}$. This gives a characterization of irreducible
modules in ${}_{H}^{H}\mathcal{YD}$.

This result has some distance from the classical result for the category
$_{kG}^{kG}\mathcal{YD}$, since $A_{1},\ldots,A_{r}$ are algebras in
$_{kG}\mathcal{M}$; whereas in \cite{dijkgraaf1992quasi,gould1993quantum} an
irreducible module of $_{kG}^{kG}\mathcal{YD}$ can be characterized by an
irreducible module of the group algebra $kC\left(  g\right)  $ of the
centralizer $C\left(  g\right)  $ of an element $g\in G$, with that $kC\left(
g\right)  $ need not to be a left $kG$-module.

If $\mathcal{C}=\mathrm{Vec}_{k}$, then $M_{i}=H\otimes W_{i}$ is still a
generator of the (possibly decomposable) module category $_{H}^{D_{i}%
}\mathcal{M}$ over $\mathcal{C}$. In this setting we can prove that
\[
A_{i}=\underline{\operatorname{Hom}}\left(  M_{i},M_{i}\right)
=\operatorname{Hom}_{H}^{D_{i}}\left(  H\otimes W_{i},H\otimes W_{i}\right)
\cong W_{i}^{\ast}\square_{D_{i}}\left(  H\otimes W_{i}\right)  \text{,}%
\]
is an ordinary algebra. We call $N_{W_{i}}=W_{i}^{\ast}\square_{D_{i}}\left(
H\otimes W_{i}\right)  $ the $R$-adjoint-stable algebra $N_{W_{i}}$ of $W_{i}%
$. Then $V$ is an irreducible object in ${}_{H}^{H}\mathcal{YD}$ if and only
if there exists some $1\leq i\leq r$, such that $V\cong U\otimes_{N_{W_{i}}%
}\left(  H\otimes W_{i}\right)  $ for an irreducible right $N_{W_{i}}$-module
$U$.

As an application of the structure theorem, we get that $N_{W_{i}}$ is an
$H$-simple comodule algebra. By the work of Skryabin
\cite{skryabin2007projectivity} if we assume further that the field $k$ is
algebraically closed, we obtain
\[
(\dim U\dim W_{i})\mid\dim N_{W_{i}},
\]
where $U$ is any irreducible right $N_{W_{i}}$-module. Therefore, $N_{W_{i}}$
is such an algebra over which the dimension of each irreducible left module
divides its dimension.

Back to the finite group algebra, if $H=kG,$ then the minimal decomposition in
${}_{H}^{H}\mathcal{YD}$ becomes $kG=k\mathcal{C}_{1}\oplus\cdots\oplus
k\mathcal{C}_{r}$ where $\left\{  \mathcal{C}_{i}\mid1\leq i\leq r\right\}  $
are conjugacy classes of $G$. Let $W_{i}=kg_{i}$, where $g_{i}\in
\mathcal{C}_{i}$. Then $kC\left(  g_{i}\right)  \cong(N_{W_{i}})^{op}$, and
the main results in this paper generalize the structure theorems appeared in
\cite{dijkgraaf1992quasi,gould1993quantum,Andruskiewitsch1998Braided}.

This paper is organized as follows. In section~\ref{section-prelimaries} we
recall some preliminaries. In section \ref{Section-struct-HR}, we recall
Majid's construction of the transmuted braided group $H_{R}$, and then present
some more properties for $H_{R}$. In Section \ref{section-mod-cat-general}, we
study the module category $_{H}^{H}\mathcal{YD}={}_{H}^{H_{R}}\mathcal{M}$ by
decomposing $_{H}^{H_{R}}\mathcal{M}$ into a sum of module subcategories%
\[
_{H}^{H_{R}}\mathcal{M}=\oplus{}_{i=1}^{r}{}_{H}^{D_{i}}\mathcal{M}\text{,}%
\]
where
\[
H_{R}=D_{1}\oplus\cdots\oplus D_{r}%
\]
is the sum of minimal $H$-adjoint-stable subcoalgebras of $H_{R}$. We give the
structure of the irreducible object $V$ in $_{H}^{D_{i}}\mathcal{M}$ in
Section \ref{section-YDmodule}. We also study the case when $D_{i}$ contains a
grouplike element, and apply our main results to the Kac-Paljutkin
$8$-dimensional semisimple Hopf algebra $H_{8}$.

\section{Preliminaries}

\label{section-prelimaries}

Throughout this paper, $k$ is a field, and all vector spaces are over $k$.
$\operatorname{Hom}$ and $\otimes$ for vector spaces are taken over $k$ if not
specified. For a subset $X$ of a vector space, we use $\mathrm{span}X$ to
denote the linear subspace spanned by $X$. If $A$ is an algebra, the notation
$_{A}\mathcal{M}$ (resp. $\mathcal{M}_{A}$) denotes the category of left
(resp. right) $A$-modules. $H$ will always denote a Hopf algebra over $k$ with
comultiplication $\Delta$, counit $\varepsilon,$ and antipode $S.$ If $S$ is
bijective, we write $\bar{S}$ for its composite inverse. We will use the
Sweedler's sigma notation \cite{MR0252485} for coproduct and coaction:
$\Delta\left(  h\right)  =\sum h_{\left(  1\right)  }\otimes h_{\left(
2\right)  }$ for coalgebras and $\rho\left(  m\right)  =\sum m_{\left\langle
0\right\rangle }\otimes m_{\left\langle 1\right\rangle }$ for right comodules
(or $\rho^{\prime}\left(  m\right)  =\sum m_{\left\langle -1\right\rangle
}\otimes m_{\left\langle 0\right\rangle }$ for left comodules). Our references
for Hopf algebras are \cite{MR0252485,Montgomery1993Hopf}.

Recall that a left-left Yetter-Drinfeld $H$-module $V$ is both a left
$H$-module and a left $H$-comodule, satisfying the condition%
\begin{equation}
\sum\left(  hv\right)  _{\langle-1\rangle}\otimes\left(  hv\right)
_{\langle0\rangle}=\sum h_{(1)}v_{\langle-1\rangle}\left(  Sh_{(3)}\right)
\otimes h_{(2)}v_{\langle0\rangle}, \label{YDcompati}%
\end{equation}
for all $h\in H,\ v\in V$. The category of left-left Yetter-Drinfeld
$H$-modules is denoted by ${}_{H}^{H}\mathcal{YD}$. Similarly, there is also a
notion of the left-right Yetter-Drinfeld module category ${}_{H}%
\mathcal{YD}^{H}$.

A pair $(H,R)$ is called a \textit{quasi-triangular} Hopf algebra
(\cite[Section 10]{drinfeld1986quantum}) if $R=\sum R^{1}\otimes R^{2}\in
H\otimes H $ is an invertible element which satisfies
\begin{align}
&  \sum R^{1}h_{(1)}\otimes R^{2}h_{(2)}=\sum h_{(2)}R^{1}\otimes h_{(1)}%
R^{2},~\forall~h\in H,\label{QT1}\\
&  (\Delta\otimes id_{H})(R)=\sum{R_{1}}^{1}\otimes{R_{2}}^{1}\otimes{R_{1}%
}^{2}{R_{2}}^{2},\label{QT2}\\
&  (id_{H}\otimes\Delta)(R)=\sum{R_{1}}^{1}{R_{2}}^{1}\otimes{R_{2}}%
^{2}\otimes{R_{1}}^{2}, \label{QT3}%
\end{align}
where $R_{i}=R=\sum{R_{i}}^{1}\otimes{R_{i}}^{2},\allowbreak\ \forall
~i\in\mathbb{N}^{+}$. Such an element $R$ is called an R-matrix of $H$.

If $\left(  H,R\right)  $ is quasi-triangular, then the following properties
hold (cf. Drinfeld~\cite{drinfeld1990almost}):

\begin{enumerate}
\item The R-matrix $R$ is a solution of the quantum Yang-Baxter equation%
\begin{equation}
R^{12}R^{13}R^{23}=R^{23}R^{13}R^{12}, \label{QYBE}%
\end{equation}
where $R^{12}=\sum R^{1}\otimes R^{2}\otimes1_{H},$ $R^{23}=\sum1_{H}\otimes
R^{1}\otimes R^{2},$ and $R^{13}=\sum R^{1}\otimes1_{H}\otimes R^{2}$.

\item $\left(  S\otimes id_{H}\right)  \left(  R\right)  =R^{-1}=\left(
id_{H}\otimes\bar{S}\right)  \left(  R\right)  $, and $\left(  S\otimes
S\right)  \left(  R\right)  =R$.

\item $\left(  \varepsilon\otimes id_{H}\right)  \left(  R\right)
=1_{H}=\left(  id_{H}\otimes\varepsilon\right)  \left(  R\right)  $.
\end{enumerate}

From now on, we assume that $H$ is a quasi-triangular Hopf algebra with
R-matrix $R$, and consider the left-left Yetter-Drinfeld modules, unless
otherwise stated. We use $R_{i}=R\ (i=1,2,\ldots)$, when more than one $R$ are used.

For $\left(  H,R\right)  ,$ Majid defined a transmuted braided group $H_{R}$,
which is a cocommutative Hopf algebra in the braided tensor category
$_{H}\mathcal{M}$. For details, one can refer to \cite{Majid1991Braided}. We
will use the explicit coalgebra expression of $H_{R}$, to yield a construction
of irreducible Yetter-Drinfeld modules.

Write the left adjoint action of $H$ on itself by $\cdot_{ad}$, namely,
$h\cdot_{ad}a=\sum h_{(1)}a\left(  Sh_{(2)}\right)  $, $\forall~h,\ a\in H$.

\begin{defn}
[\cite{Majid1991Braided}]Let $\left(  H,R\right)  $ be a quasi-triangular Hopf
algebra. Then there is a Hopf algebra $H_{R}$ in the braided tensor category
$_{H}\mathcal{M}$, defined as follows. $H_{R}=H$ as an algebra, and the
$H$-module action is $\cdot_{ad}$. The comultiplication and antipode is
defined by%
\begin{align}
\Delta_{R}(h)  &  =\sum h_{(1)}\left(  SR^{2}\right)  \otimes R^{1}\cdot
_{ad}h_{(2)},\label{DefDeltaR}\\
S_{R}\left(  h\right)   &  =\sum R^{2}S\left(  R^{1}\cdot_{ad}h\right)  ,
\end{align}
for all $h\in H$. This $H_{R}$ is called the transmuted braided group of
$\left(  H,R\right)  .$
\end{defn}

To avoid confusion, we write $\Delta_{R}(h)=\sum h^{(1)}\otimes h^{(2)}$. Note
that $\Delta_{R}$ is a morphism in $_{H}\mathcal{M},$ that is, for all $h,a\in
H,$%
\begin{equation}
\Delta_{R}\left(  h\cdot_{ad}a\right)  =\sum h_{(1)}\cdot_{ad}a^{(1)}\otimes
h_{(2)}\cdot_{ad}a^{(2)}. \label{eqc}%
\end{equation}

\begin{lem}
[{\cite[Lemma 4.7]{Zhu2015Relative}}]\label{lemHHYDeqHRHM}

\begin{enumerate}
\item Let $V$ be an object of ${}_{H}^{H}\mathcal{YD},$ then the map $\rho
_{R}:V\rightarrow H_{R}\otimes V$ given by
\begin{equation}
\rho_{R}(v)=\sum v_{\left\langle -1\right\rangle }\left(  SR^{2}\right)
\otimes R^{1}v_{\left\langle 0\right\rangle },\ v\in V \label{eqa}%
\end{equation}
is a left $H_{R}$-comodule map.

\item If $V$ is a left $H$-module, and at the same time $\left(  V,\rho
_{R}\right)  $ is a left $H_{R}$-comodule with $\rho_{R}\left(  v\right)
=\sum v^{\langle-1\rangle}\otimes v^{\langle0\rangle}$ for $v\in V$. Define
$\rho\left(  v\right)  =\sum v^{\left\langle -1\right\rangle }R^{2}\otimes
R^{1}v^{\left\langle 0\right\rangle }\in H\otimes V$ for $v\in V$, then
$V\in{}_{H}^{H}\mathcal{YD}$ if and only if for any $v\in V$ and $h\in H$,
\begin{equation}
\rho_{R}(hv)=\sum h_{(1)}\cdot_{ad}v^{\langle-1\rangle}\otimes h_{(2)}%
v^{\langle0\rangle}. \label{eqb}%
\end{equation}

\end{enumerate}

The category $_{H}^{H}\mathcal{YD}$ and the category $_{H}^{H_{R}}\mathcal{M}
$ of $H_{R}$-comodules in $_{H}\mathcal{M}$ are identical.
\end{lem}

Since the transmuted braided group $H_{R}$ is a coalgebra in $_{H}\mathcal{M}%
$, the category ${}_{H}^{H}\mathcal{YD}={}_{H}^{H_{R}}\mathcal{M}$ can be
viewed as a module category over the monoidal category {}$_{H^{cop}%
}\mathcal{M}$. It will be explained in detail in Section
\ref{section-mod-cat-general}.

The general theory of module categories over a monoidal category was developed
by Ostrik. For references, one can see
\cite{Ostrik2003module,Etingof2015tensor}. Later on, all categories we
considered will be at least $k$-linear abelian.

Assume that $\mathcal{C}$ is a fusion category, and $\left(  \mathcal{M}%
,\otimes,a\right)  $ is a semisimple module category over $\mathcal{C}$, with
$a$ being the module associativity functorial isomorphism. For two objects
$M_{1},M_{2}$ of $\mathcal{M}$, the \textit{internal} Hom of $M_{1}$ and
$M_{2}$ is an object $\underline{\operatorname{Hom}}\left(  M_{1}%
,M_{2}\right)  $ of $\mathcal{C}$ representing the functor $X\mapsto
\operatorname{Hom}_{\mathcal{M}}\left(  X\otimes M_{1},M_{2}\right)
:\mathcal{C}\rightarrow\mathrm{Vec}_{k}$. It implies that there exists a
natural isomorphism%

\begin{equation}
\eta_{\bullet,M_{1},M_{2}}:\operatorname{Hom}_{\mathcal{M}}\left(
\bullet\otimes M_{1},M_{2}\right)  \overset{\cong}{\longrightarrow
}\operatorname{Hom}_{\mathcal{C}}\left(  \bullet,\underline{\operatorname{Hom}%
}\left(  M_{1},M_{2}\right)  \right)  . \label{def-in-hom}%
\end{equation}
Let $ev_{M_{1},M_{2}}:\underline{\operatorname{Hom}}\left(  M_{1}%
,M_{2}\right)  \otimes M_{1}\rightarrow M_{2}$ be the evaluation morphism
obtained as the image of the $id_{\underline{\operatorname{Hom}}\left(
M_{1},M_{2}\right)  }$ under the isomorphism
\[
\operatorname{Hom}_{\mathcal{C}}\left(  \underline{\operatorname{Hom}}\left(
M_{1},M_{2}\right)  ,\underline{\operatorname{Hom}}\left(  M_{1},M_{2}\right)
\right)  \overset{\cong}{\longrightarrow}\operatorname{Hom}_{\mathcal{M}%
}\left(  \underline{\operatorname{Hom}}\left(  M_{1},M_{2}\right)  \otimes
M_{1},M_{2}\right)  .
\]
If $M_{1},M_{2},M_{3}$ are objects of $\mathcal{M}$, then the multiplication
(composition) of internal Hom
\[
\mu_{M_{1},M_{2},M_{3}}:\underline{\operatorname{Hom}}\left(  M_{2}%
,M_{3}\right)  \otimes\underline{\operatorname{Hom}}\left(  M_{1}%
,M_{2}\right)  \rightarrow\underline{\operatorname{Hom}}\left(  M_{1}%
,M_{3}\right)
\]
is defined as the image of the morphism%

\[
ev_{M_{2},M_{3}}\circ\left(  id\otimes ev_{M_{1},M_{2}}\right)  \circ
a_{\underline{\operatorname{Hom}}\left(  M_{2},M_{3}\right)  ,\underline
{\operatorname{Hom}}\left(  M_{1},M_{2}\right)  ,M_{1}}%
\]
under the isomorphism
\begin{align*}
&  \operatorname{Hom}_{\mathcal{M}}\left(  \left(  \underline
{\operatorname{Hom}}\left(  M_{2},M_{3}\right)  \otimes\underline
{\operatorname{Hom}}\left(  M_{1},M_{2}\right)  \right)  \otimes M_{1}%
,M_{3}\right) \\
&  \cong\operatorname{Hom}_{\mathcal{C}}\left(  \underline{\operatorname{Hom}%
}\left(  M_{2},M_{3}\right)  \otimes\underline{\operatorname{Hom}}\left(
M_{1},M_{2}\right)  ,\underline{\operatorname{Hom}}\left(  M_{1},M_{3}\right)
\right)  .
\end{align*}
Then for any $M,V\in\mathcal{M}$, the internal Hom $\left(  \underline
{\operatorname{Hom}}\left(  M,M\right)  ,\mu_{M,M,M}\right)  =A$ is an algebra
in $\mathcal{C}$ with unit $u_{M}:1\rightarrow\underline{\operatorname{Hom}%
}\left(  M,M\right)  $ obtained from isomorphism $\operatorname{Hom}%
_{\mathcal{M}}\left(  M,M\right)  $ $\overset{\cong}{\longrightarrow}$
$\operatorname{Hom}_{\mathcal{C}}\left(  1,\underline{\operatorname{Hom}%
}\left(  M,M\right)  \right)  $ as the image of $id$; and $\left(
\underline{\operatorname{Hom}}\left(  M,V\right)  ,\mu_{M,M,V}\right)  $ is a
right $A$-module in $\mathcal{C}$.

Ostrik characterized in \cite[Theorem 3.1]{Ostrik2003module} (especially
indecomposable) module categories over $\mathcal{C}$ in the following theorem.

\begin{lem}
[{cf. \cite[Theorem 3.1]{Ostrik2003module},\cite[Theorem 7.10.1]%
{Etingof2015tensor}}]\label{lemOstrik}Let $\mathcal{M}$ be a semisimple
category over a fusion category $\mathcal{C}$. If $M\in\mathcal{M}$ is a
generator, then $A=\underline{\operatorname{Hom}}\left(  M,M\right)  $ is a
semisimple algebra in $\mathcal{C}$. The functor $F=\underline
{\operatorname{Hom}}\left(  M,\bullet\right)  :\mathcal{M}\rightarrow
\mathrm{Mod}_{\mathcal{C}}(A)$ given by $V\mapsto\underline{\operatorname{Hom}%
}\left(  M,V\right)  $ is an equivalence of $\mathcal{C}$-module categories.

If further that $\mathcal{M}$ is indecomposable, then every nonzero object $M$
generates $\mathcal{M}$, and the functor $F=\underline{\operatorname{Hom}%
}\left(  M,\bullet\right)  :\mathcal{M}\rightarrow\mathrm{Mod}_{\mathcal{C}%
}(A)$ given by $V\mapsto\underline{\operatorname{Hom}}\left(  M,V\right)  $ is
an equivalence of $\mathcal{C}$-module categories.
\end{lem}

\section{The Transmuted Braided Group}

\label{Section-struct-HR}

We start this section with an arbitrary quasi-triangular Hopf algebra $(H,R)$.
The cosemisimplicity of the transmuted braided group $\left(  H_{R},\Delta
_{R}\right)  $ will be discussed. If $H$ is cosemisimple and unimodular, we
prove that $H_{R}$ is cosemisimple as a $k$-coalgebra. Naturally, $\left(
H,\cdot_{ad},\Delta\right)  \in{}_{H}^{H}\mathcal{YD}$. By Lemma
\ref{lemHHYDeqHRHM}, $H$ is an object of $_{H}^{H_{R}}\mathcal{M}$ via the
$H_{R}$-coaction $\Delta_{R}$ and the $H$-action $\cdot_{ad}.$ If further we
assume that $H$ is semisimple and cosemisimple, then $_{H}^{H_{R}}\mathcal{M}$
is semisimple, and in $_{H}^{H_{R}}\mathcal{M}$ the object $H$ can be
decomposed as a direct sum of minimal (irreducible) sub-objects.

\begin{prop}
\label{PropSubmodofH}For any Yetter-Drinfeld submodule $D$ of $H$, $D$ is a
subcoalgebra of $H_{R}$. Moreover, $D$ is a minimal Yetter-Drinfeld submodule
of $H$ if and only if $D$ is a minimal $H$-adjoint-stable subcoalgebra of
$H_{R}$ with the left adjoint action $\cdot_{ad}$.
\end{prop}

\begin{proof}
Since $\left(  D,\cdot_{ad},\Delta\right)  \in{}_{H}^{H}\mathcal{YD}={}%
_{H}^{H_{R}}\mathcal{M}$ , we obtain that $\Delta_{R}\left(  D\right)
\subseteq$ $H\otimes D$, $H\cdot_{ad}D\subseteq D$ and $\Delta\left(
D\right)  \subseteq H\otimes D$. By $\left(  S\otimes S\right)  \left(
R\right)  =R$, we have
\begin{align*}
\Delta_{R}(h)  &  =\sum h_{(1)}\left(  SR^{2}\right)  \otimes R^{1}\cdot
_{ad}h_{(2)}\\
&  =\sum h_{(1)}S\left(  {R_{1}}^{2}{R_{2}}^{2}\right)  \otimes{R_{1}}%
^{1}h_{(2)}\left(  S{R_{2}}^{1}\right) \\
&  =\sum h_{(1)}{R_{2}}^{2}\left(  S{R_{1}}^{2}\right)  \otimes{R_{1}}%
^{1}h_{(2)}{R_{2}}^{1}\\
&  =\sum{R_{2}}^{2}h_{(2)}\left(  S{R_{1}}^{2}\right)  \otimes{R_{1}}%
^{1}{R_{2}}^{1}h_{(1)}\\
&  =\sum{R}^{2}\cdot_{ad}h_{(2)}\otimes R^{1}h_{(1)}\in D\otimes H,
\end{align*}
for all $h\in D$. Therefore we also have $\Delta_{R}\left(  D\right)
\subseteq$ $D\otimes H,$ and so $D$ is a subcoalgebra of $H_{R}.$ Conversely,
by Lemma \ref{lemHHYDeqHRHM}, any $H$-adjoint-stable subcoalgebra of $H_{R}$
is a Yetter-Drinfeld submodule of $H$, and the proposition follows.
\end{proof}

Assume that $H$ is finite dimensional. Let $\Lambda$ be a left integral for
$H$, $\lambda$ a right integral for $H^{\ast}$, and $\mathbf{a}$, $\alpha$ the
distinguished grouplike elements of $H$ and $H^{\ast}$ respectively. It is
known from \cite[Theorem 3]{RADFORD1994583} that $\sum\Lambda_{(2)}%
\otimes\Lambda_{(1)}=\sum\Lambda_{(1)}\otimes\left(  S^{2}\Lambda
_{(2)}\right)  \mathbf{a}$. If $H^{\ast}$ is unimodular and $S^{2}=id_{H}$, it
follows that $\sum\Lambda_{(1)}\otimes\Lambda_{(2)}=\sum\Lambda_{(2)}%
\otimes\Lambda_{(1)}$.

For the coalgebra $H_{R}$, we have a similar result. Set $\mathbf{u}%
=\sum\left(  SR^{2}\right)  R^{1}$ to be the Drinfeld element, and
$\tilde{\alpha}=(\alpha\otimes id_{H})R=\left(  id_{H}\otimes\alpha
^{-1}\right)  R$.

\begin{prop}
\label{propd}Let $\left(  H,R\right)  $ be a finite dimensional
quasi-triangular Hopf algebra. Then for any left integral $\Lambda$ of $H$,
$\sum\Lambda^{(1)}\otimes\Lambda^{(2)}=\sum\Lambda^{(2)}\otimes\Lambda^{(1)}$
if and only if $\sum\Lambda_{(1)}\left(  S\mathbf{u}\right)  \tilde{\alpha
}\otimes\Lambda_{(2)}=\sum\tilde{\alpha}\left(  S\mathbf{u}\right)
\Lambda_{(1)}\otimes\Lambda_{(2)}$.

In particular, if $H$ is unimodular and $S^{2}=id_{H}$, then $\Lambda$ is
cocommutative under $\Delta_{R}$.
\end{prop}

\begin{proof}
Since $\Lambda\in\int_{H}^{l}$, for any $h\in H$,
\[
\sum\Lambda_{(1)}\otimes h\Lambda_{(2)}=\sum\left(  Sh\right)  \Lambda
_{(1)}\otimes\Lambda_{(2)},
\]
and
\[
\sum\Lambda_{(1)}\otimes\Lambda_{(2)}h=\sum\Lambda_{(1)}\bar{S}\left(
\alpha\rightharpoonup h\right)  \otimes\Lambda_{(2)}.
\]
Hence,
\begin{align}
\sum\Lambda^{(1)}\otimes\Lambda^{(2)}  &  =\sum\Lambda_{(1)}S\left(  {R_{1}%
}^{2}{R_{2}}^{2}\right)  \otimes{R_{1}}^{1}\Lambda_{(2)}\left(  S{R_{2}}%
^{1}\right) \label{eqd1}\\
&  =\sum\left(  S{R_{1}}^{1}\right)  \Lambda_{(1)}\bar{S}\left(
\alpha\rightharpoonup\left(  S{R_{2}}^{1}\right)  \right)  S\left(  {R_{1}%
}^{2}{R_{2}}^{2}\right)  \otimes\Lambda_{(2)}\nonumber\\
&  =\sum{R_{1}}^{1}\Lambda_{(1)}\bar{S}\left(  \alpha\rightharpoonup{R_{2}%
}^{1}\right)  {R_{2}}^{2}{R_{1}}^{2}\otimes\Lambda_{(2)}\nonumber\\
&  =\sum{R_{1}}^{1}\Lambda_{(1)}\left(  \bar{S}{R_{2}}^{1}\right)  {R_{2}}%
^{2}\left\langle \alpha,{R_{3}}^{1}\right\rangle {R_{3}}^{2}{R_{1}}^{2}%
\otimes\Lambda_{(2)}\nonumber\\
&  =\sum{R}^{1}\Lambda_{(1)}\left(  S\mathbf{u}\right)  \tilde{\alpha}{R}%
^{2}\otimes\Lambda_{(2)}.\nonumber
\end{align}
On the other hand, by (\ref{QT1})
\begin{align}
\sum\Lambda^{(2)}\otimes\Lambda^{(1)}  &  =\sum{R_{1}}^{1}\Lambda_{(2)}{R_{2}%
}^{1}\otimes\Lambda_{(1)}{R_{2}}^{2}\left(  S{R_{1}}^{2}\right) \label{eqd2}\\
&  =\sum{R_{1}}^{1}{R_{2}}^{1}\Lambda_{(1)}\otimes{R_{2}}^{2}\Lambda
_{(2)}\left(  S{R_{1}}^{2}\right) \nonumber\\
&  =\sum{R_{1}}^{1}{R_{2}}^{1}\left(  S{R_{2}}^{2}\right)  \Lambda_{(1)}%
\bar{S}\left(  \alpha\rightharpoonup\left(  S{R_{1}}^{2}\right)  \right)
\otimes\Lambda_{(2)}\nonumber\\
&  =\sum{R_{1}}^{1}{R_{2}}^{1}\left\langle \alpha,S{R_{2}}^{2}\right\rangle
\left(  S\mathbf{u}\right)  \Lambda_{(1)}{R_{1}}^{2}\otimes\Lambda
_{(2)}\nonumber\\
&  =\sum{R}^{1}\tilde{\alpha}\left(  S\mathbf{u}\right)  \Lambda_{(1)}{R}%
^{2}\otimes\Lambda_{(2)}.\nonumber
\end{align}
Notice that $\sum{R_{2}}^{1}{R_{1}}^{1}\otimes{R_{1}}^{2}\left(  S{R_{2}}%
^{2}\right)  =\sum{R_{2}}^{1}\left(  S{R_{1}}^{1}\right)  \otimes S\left(
{R_{2}}^{2}{R_{1}}^{2}\right)  =1_{H}\otimes1_{H}.$ Comparing (\ref{eqd1})
with (\ref{eqd2}), we get the first part of this proposition. As illustrated
by Drinfeld \cite[Proposition 2.1]{drinfeld1990almost}, $S^{2}h=\mathbf{u}%
h\mathbf{u}^{-1}=S\left(  \mathbf{u}^{-1}\right)  h\left(  S\mathbf{u}\right)
$, for all $h\in H$. Then $S^{2}=id_{H}$ implies that $S\mathbf{u}$ is in the
center of $H$. Therefore, if $H$ is unimodular and $S^{2}=id_{H}$, then
$\sum\Lambda^{(1)}\otimes\Lambda^{(2)}=\sum\Lambda^{(2)}\otimes\Lambda^{(1)}$.
\end{proof}

Let ${H_{R}}^{\ast}=(H^{\ast},\ast_{R},\varepsilon)$ be the convolution
algebra of $H_{R}$. By (\ref{DefDeltaR}),
\begin{equation}
f\ast_{R}g=\sum(SR^{2}\rightharpoonup f)\ast\left(  g\leftharpoonup
\!\!\!\!\leftharpoonup R^{1}\right)  ,\ \forall~f,g\in{H_{R}}^{\ast},
\label{def*R}%
\end{equation}
where $\leftharpoonup\!\!\!\!\leftharpoonup$ denotes the right $H$-coadjoint
action on $H^{\ast}$ (i.e. $f\leftharpoonup\!\!\!\!\leftharpoonup
h=\sum(Sh_{(2)})\rightharpoonup f\leftharpoonup h_{(1)}$, $\forall~h\in
H,\ f\in H^{\ast}$). Thus, (\ref{eqc}) implies that
\begin{equation}
\left(  f\ast_{R}g\right)  \leftharpoonup\!\!\!\!\leftharpoonup h=\sum\left(
f\leftharpoonup\!\!\!\!\leftharpoonup h_{(1)}\right)  \ast_{R}\left(
g\leftharpoonup\!\!\!\!\leftharpoonup h_{(2)}\right)  ,\forall~h\in H,f,g\in
H^{\ast}. \label{*R_property}%
\end{equation}

\begin{lem}
[{{\cite[Theorem 3]{RADFORD1994583}}}]\label{lemb} Let $\lambda$ be an
integral of $H^{\ast},$ then for any $h\in H $,

\begin{enumerate}
\item $\sum\lambda_{(1)}\otimes h\rightharpoonup S^{\ast}\lambda_{(2)}=\sum
Sh\rightharpoonup\lambda_{(1)}\otimes S^{\ast}\lambda_{(2)}$,

\item $\sum\lambda_{(1)}\otimes S^{\ast}\lambda_{(2)}\leftharpoonup
h=\sum\lambda_{(1)}\leftharpoonup S^{3}h_{(1)}\left\langle \alpha^{-1}%
,h_{(2)}\right\rangle \otimes S^{\ast}\lambda_{(2)}$.
\end{enumerate}
\end{lem}

Clearly, $H_{R}$ is cosemisimple if and only if ${H_{R}}^{\ast}$ is
semisimple. We next find a separable idempotent of ${H_{R}}^{\ast}$, provided
that ${H}$ is cosemisimple and unimodular.

\begin{prop}
\label{propcos.s}Let $\left(  H,R\right)  $ be a finite dimensional
quasi-triangular Hopf algebra. Assume that $H$ is cosemisimple and unimodular.
Let $\lambda$ be an integral of $H^{\ast}$ which satisfies $\langle
\lambda,1_{H}\rangle=1$, and $e_{R}=\sum R^{2}\rightharpoonup\lambda
_{(1)}\otimes S^{\ast}\lambda_{(2)}\leftharpoonup\!\!\!\!\leftharpoonup R^{1}
$. Then ${H_{R}}^{\ast}$ is semisimple and $e_{R}$ is a separable idempotent, i.e.

\begin{enumerate}
\item $\sum\left(  R^{2}\rightharpoonup\lambda_{(1)}\right)  \ast_{R}\left(
S^{\ast}\lambda_{(2)}\leftharpoonup\!\!\!\!\leftharpoonup R^{1}\right)
=\varepsilon$,

\item For all $f\in H^{\ast}$, $\left(  f\otimes\varepsilon\right)  \ast
_{R}e_{R}=e_{R}\ast_{R}\left(  \varepsilon\otimes f\right)  $.
\end{enumerate}
\end{prop}

\begin{proof}
Part 1) follows immediately by (\ref{def*R}). Since $H$ and $H^{\ast}$ are
unimodular, it is clearly that $\mathbf{u}S\left(  \mathbf{u}^{-1}\right)
=\mathbf{g}=\mathbf{a}\tilde{\alpha}=1_{H}.$ By Lemma \ref{lemb}, we have
\begin{align*}
e_{R}  &  =\sum{R_{1}}^{2}{R_{2}}^{2}\rightharpoonup\lambda_{(1)}\otimes
S{R_{2}}^{1}\rightharpoonup S^{\ast}\lambda_{(2)}\leftharpoonup{R_{1}}^{1}\\
&  =\sum{R_{1}}^{2}{R_{2}}^{2}\left(  S^{2}{R_{2}}^{1}\right)  \rightharpoonup
\lambda_{(1)}\otimes S^{\ast}\lambda_{(2)}\leftharpoonup{R_{1}}^{1}\\
&  =\sum{R_{1}}^{2}{R_{2}}^{2}u^{-1}\rightharpoonup\lambda_{(1)}\leftharpoonup
S^{3}{R_{1}}^{1}\left\langle \alpha^{-1},{R_{2}}^{1}\right\rangle \otimes
S^{\ast}\lambda_{(2)}\\
&  =\sum{R}^{2}\tilde{\alpha}^{-1}u^{-1}\rightharpoonup\lambda_{(1)}%
\leftharpoonup S^{3}{R}^{1}\otimes S^{\ast}\lambda_{(2)}\\
&  =\sum S{R}^{2}u^{-1}\rightharpoonup\lambda_{(1)}\leftharpoonup{R}%
^{1}\otimes S^{\ast}\lambda_{(2)}.
\end{align*}
Then for $f\in H^{\ast}$, we get
\begin{align*}
&  \left(  f\otimes\varepsilon\right)  \ast_{R}e_{R}\\
&  =\sum f\ast_{R}\left(  \left(  S{R}^{2}\right)  u^{-1}\rightharpoonup
\lambda_{(1)}\leftharpoonup{R}^{1}\right)  \otimes S^{\ast}\lambda_{(2)}\\
&  =\sum\left(  S{R_{2}}^{2}\rightharpoonup f\right)  \ast\left(  \left(
\left(  S{R_{1}}^{2}\right)  u^{-1}\rightharpoonup\lambda_{(1)}\leftharpoonup
{R_{1}}^{1}\right)  \leftharpoonup\!\!\!\!\leftharpoonup{R_{2}}^{1}\right)
\otimes S^{\ast}\lambda_{(2)}\\
&  =\sum\left(  S\left(  {R_{2}}^{2}{R_{3}}^{2}\right)  \rightharpoonup
f\right)  \ast\left(  S\left(  {R_{1}}^{2}{R_{3}}^{1}\right)  u^{-1}%
\rightharpoonup\lambda_{(1)}\leftharpoonup{R_{1}}^{1}{R_{2}}^{1}\right)
\otimes S^{\ast}\lambda_{(2)}\\
&  =\sum\left(  S\left(  {R_{2}}^{2}{R_{1}}^{2}\right)  \rightharpoonup
f\right)  \ast\left(  S\left(  {R_{2}}^{1}{R_{3}}^{2}\right)  u^{-1}%
\rightharpoonup\lambda_{(1)}\leftharpoonup{R_{1}}^{1}{R_{3}}^{1}\right)
\otimes S^{\ast}\lambda_{(2)}\\
&  =\sum\left(  {R_{1}}^{2}{R_{2}}^{2}\rightharpoonup f\right)  \ast\left(
{R_{3}}^{2}{R_{2}}^{1}u^{-1}\rightharpoonup\lambda_{(1)}\leftharpoonup\bar
{S}\left(  {R_{3}}^{1}{R_{1}}^{1}\right)  \right)  \otimes S^{\ast}%
\lambda_{(2)}.
\end{align*}
By a result of Lyubashenko \cite{lyubashenko1986superanalysis} that
$\Delta(\mathbf{u})=\left(  {R^{21}R}\right)  ^{-1}(\mathbf{u}\otimes
\mathbf{u})=(\mathbf{u}\otimes\mathbf{u}){(R^{21}R)}^{-1}$, then
\begin{align*}
&  e_{R}\ast_{R}\left(  \varepsilon\otimes f\right) \\
&  =\sum\left(  S{R}^{2}\right)  u^{-1}\rightharpoonup\lambda_{(1)}%
\leftharpoonup{R}^{1}\otimes S^{\ast}\lambda_{(2)}\ast_{R}f\\
&  =\sum\left(  S{R_{1}}^{2}\right)  u^{-1}\rightharpoonup\lambda
_{(1)}\leftharpoonup{R_{1}}^{1}\otimes\left(  S{R_{2}}^{2}\rightharpoonup
S^{\ast}\lambda_{(2)}\right)  \ast\left(  f\leftharpoonup
\!\!\!\!\leftharpoonup{R_{2}}^{1}\right) \\
&  =\sum\left(  S{R_{1}}^{2}\right)  u^{-1}\left(  S^{2}{R_{2}}^{2}\right)
\rightharpoonup\lambda_{(1)}\leftharpoonup{R_{1}}^{1}\otimes\left(  S^{\ast
}\lambda_{(2)}\right)  \ast\left(  f\leftharpoonup\!\!\!\!\leftharpoonup
{R_{2}}^{1}\right) \\
&  =\sum\left(  S{R_{1}}^{2}\right)  {R_{2}}^{2}u^{-1}\rightharpoonup\left(
\left(  f\leftharpoonup\!\!\!\!\leftharpoonup{R_{2}}^{1}\right)  \ast
\lambda_{(1)}\right)  \leftharpoonup{R_{1}}^{1}\otimes S^{\ast}\lambda_{(2)}\\
&  =\sum\left(  S{R_{1}}^{2}\right)  {R_{2}}^{2}{R_{3}}^{2}u^{-1}%
\rightharpoonup\left(  \left(  S{R_{3}}^{1}\rightharpoonup f\leftharpoonup
{R_{2}}^{1}\right)  \ast\lambda_{(1)}\right)  \leftharpoonup{R_{1}}^{1}\otimes
S^{\ast}\lambda_{(2)}\\
&  =\sum S\left(  {R_{1}}^{2}{R_{4}}^{2}\right)  {R_{2}}^{2}{R_{3}}^{2}%
u^{-1}\rightharpoonup\left(  \left(  S{R_{3}}^{1}\rightharpoonup
f\leftharpoonup{R_{2}}^{1}{R_{1}}^{1}\right)  \ast\left(  \lambda
_{(1)}\leftharpoonup{R_{4}}^{1}\right)  \right)  \otimes S^{\ast}\lambda
_{(2)}\\
\end{align*}
\begin{align*}
&  =\sum S\left(  {R_{2}}^{2}{R_{1}}^{2}{R_{4}}^{2}\right)  {R_{3}}^{2}%
u^{-1}\rightharpoonup\left(  \left(  S{R_{3}}^{1}\rightharpoonup
f\leftharpoonup\left(  S{R_{2}}^{1}\right)  {R_{1}}^{1}\right)  \ast\left(
\lambda_{(1)}\leftharpoonup{R_{4}}^{1}\right)  \right)  \otimes S^{\ast
}\lambda_{(2)}\\
&  =\sum{R_{1}}^{2}{R_{2}}^{2}u^{-1}\rightharpoonup\left(  \left(  S{R_{2}%
}^{1}\rightharpoonup f\right)  \ast\left(  \lambda_{(1)}\leftharpoonup\bar
{S}{R_{1}}^{1}\right)  \right)  \otimes S^{\ast}\lambda_{(2)}\\
&  =\sum\left(  {R_{1}}^{2}{R_{5}}^{2}{R_{3}}^{2}{R_{4}}^{1}u^{-1}S\left(
{R_{2}}^{1}{R_{5}}^{1}\right)  \rightharpoonup f\right) \\
&  \quad\ast\left(  {R_{6}}^{2}{R_{2}}^{2}{R_{3}}^{1}{R_{4}}^{2}%
u^{-1}\rightharpoonup\lambda_{(1)}\leftharpoonup\bar{S}\left(  {R_{6}}%
^{1}{R_{1}}^{1}\right)  \right)  \otimes S^{\ast}\lambda_{(2)}\\
&  =\sum\left(  {R_{1}}^{2}{R_{2}}^{2}{R_{3}}^{2}{R_{4}}^{1}u^{-1}S\left(
{R_{3}}^{1}{R_{5}}^{1}\right)  \rightharpoonup f\right) \\
&  \quad\ast\left(  {R_{6}}^{2}{R_{2}}^{1}{R_{5}}^{2}{R_{4}}^{2}%
u^{-1}\rightharpoonup\lambda_{(1)}\leftharpoonup\bar{S}\left(  {R_{6}}%
^{1}{R_{1}}^{1}\right)  \right)  \otimes S^{\ast}\lambda_{(2)}\\
&  =\sum\left(  {R_{1}}^{2}{R_{2}}^{2}{R_{3}}^{1}{R_{5}}^{2}u^{-1}S\left(
{R_{4}}^{1}{R_{5}}^{1}\right)  \rightharpoonup f\right) \\
&  \quad\ast\left(  {R_{6}}^{2}{R_{2}}^{1}{R_{3}}^{2}{R_{4}}^{2}%
u^{-1}\rightharpoonup\lambda_{(1)}\leftharpoonup\bar{S}\left(  {R_{6}}%
^{1}{R_{1}}^{1}\right)  \right)  \otimes S^{\ast}\lambda_{(2)}\\
&  =\sum\left(  {R_{1}}^{2}{R_{2}}^{2}{R_{3}}^{1}u^{-1}\left(  S^{2}{R_{5}%
}^{2}\right)  S\left(  {R_{4}}^{1}{R_{5}}^{1}\right)  \rightharpoonup f\right)
\\
&  \quad\ast\left(  {R_{6}}^{2}{R_{2}}^{1}{R_{3}}^{2}{R_{4}}^{2}%
u^{-1}\rightharpoonup\lambda_{(1)}\leftharpoonup\bar{S}\left(  {R_{6}}%
^{1}{R_{1}}^{1}\right)  \right)  \otimes S^{\ast}\lambda_{(2)}\\
&  =\sum\left(  {R_{1}}^{2}{R_{2}}^{2}\rightharpoonup f\right)  \ast\left(
{R_{3}}^{2}{R_{2}}^{1}u^{-1}\rightharpoonup\lambda_{(1)}\leftharpoonup\bar
{S}\left(  {R_{3}}^{1}{R_{1}}^{1}\right)  \right)  \otimes S^{\ast}%
\lambda_{(2)}\\
&  =\left(  f\otimes\varepsilon\right)  \ast_{R}e_{R}.\qedhere
\end{align*}

\end{proof}

Let $\left(  H,R\right)  $ be a semisimple and cosemisimple quasi-triangular
Hopf algebra, then by Proposition \ref{propcos.s} the transmuted braided group
$H_{R}$ is cosemisimple as a $k$-coalgebra. In addition, the Yetter-Drinfeld
module category ${}_{H}^{H}\mathcal{YD}$ is semisimple. Note that $H_{R}$ is
an $H$-module coalgebra via $\cdot_{ad}$. On one hand, as a
Yetter-Drinfeld~module $H\in{}_{H}^{H}\mathcal{YD}$ is completely reducible.
On the other hand, as an $H$-module coalgebra, $H_{R}$ can be decomposed into
a sum of minimal $H$-adjoint-stable subcoalgebras. We now give a refinement of
Proposition \ref{PropSubmodofH}.

\begin{prop}
\label{propHdecomp}Let $\left(  H,R\right)  $ be a semisimple and cosemisimple
quasi-triangular Hopf algebra, then there is a unique decomposition
\[
H_{R}=D_{1}\oplus\cdots\oplus D_{r}%
\]
of the minimal $H$-adjoint-stable subcoalgebras $D_{1},\ldots,D_{r}$ of
$H_{R}$. It is also the decomposition of $H\in{}_{H}^{H}\mathcal{YD}$ as a
direct sum of irreducible Yetter-Drinfeld modules.
\end{prop}

\section{The structure of module categories}

\label{section-mod-cat-general}Let $\mathcal{C}$ be a fusion category and
$\left(  \mathcal{M},\otimes,a\right)  $ be a $k$-linear abelian semisimple
left module category over $\mathcal{C}$. Let $\left(  A,m,u\right)  $ be a
$\mathcal{C}$-algebra. A left $A$-module in $\mathcal{M}$ is a pair $\left(
M,p\right)  $, where $M$ is an object of $\mathcal{M}$ and $p:A\otimes
M\rightarrow M$ is a morphism satisfying two natural axioms,
\[
p\circ\left(  m\otimes id\right)  =p\circ\left(  id\otimes p\right)  \circ
a_{A,A,M},\quad p\circ\left(  u\otimes id\right)  =id.
\]
As defined in \cite{Etingof2015tensor}, for right $A$-module $U$ and left
$A$-module $M$ both in $\mathcal{C}$, we can define $U\otimes_{A}M$ for right
$A$-module $\left(  U,q\right)  $ in $\mathcal{C}$ and left $A$-module
$\left(  M,p\right)  $ in $\mathcal{M}$, where $M$ is not necessarily an
object of $\mathcal{C}$.

Namely, $U\otimes_{A}M\in\mathcal{M}$ is the co-equalizer of the diagram
\[
(U\otimes A)\otimes
M\;\tikz[baseline=-.3ex] {\draw[->] (0,.8ex) -- node[above]{$q\otimes id_M$}(3cm,0.8ex); \draw[->] (0,0ex) -- node[below]{$(id_U\otimes p) \circ a_{U,A,M}$}(3cm,0ex);}\;U\otimes
M\longrightarrow U\otimes_{A}M,
\]
i.e., the cokernel of the morphism $q\otimes id_{M}-\left(  id_{U}\otimes
p\right)  \circ a_{U,A,M}$.

For any $M,V\in\mathcal{M}$, the internal Hom $\underline{\operatorname{Hom}%
}\left(  M,V\right)  $ always exists. Take $A=\underline{\operatorname{Hom}%
}\left(  M,M\right)  $, then $A$ is an algebra in $\mathcal{C}$. $\left(
\underline{\operatorname{Hom}}\left(  M,V\right)  ,\mu_{M,M,V}\right)  $ is a
right $A$-module in $\mathcal{C}$, and $\left(  M,ev_{M,M}\right)  $ is a left
$A$-module in $\mathcal{M}$, as stated in the preliminaries.

\begin{prop}
\label{propiHom&otA}Let $M$ be an object of $\mathcal{M}$, and $A=\underline
{\operatorname{Hom}}\left(  M,M\right)  $. Then for any right $A$-module
$\left(  U,q\right)  $ in $\mathcal{C}$ and any $V\in\mathcal{M}$, there is a
canonical isomorphism
\begin{equation}
\zeta_{U,V}:\operatorname{Hom}_{\mathcal{M}}\left(  U\otimes_{A}M,V\right)
\rightarrow\operatorname{Hom}_{A}\left(  U,\underline{\operatorname{Hom}%
}\left(  M,V\right)  \right)  , \label{isoadjoint}%
\end{equation}
which are natural both in $U$ and $V$.
\end{prop}

\begin{proof}
Let $\eta=\eta_{\bullet,M,V}:\operatorname{Hom}_{\mathcal{M}}\left(
\bullet\otimes M,V\right)  \overset{\cong}{\longrightarrow}\operatorname{Hom}%
_{\mathcal{C}}\left(  \bullet,\underline{\operatorname{Hom}}\left(
M,V\right)  \right)  $ be the natural isomorphism (\ref{def-in-hom}). We will
show that the required isomorphism $\zeta_{U,V}$ can be deduced from the
composition
\[
\operatorname{Hom}_{\mathcal{M}}\left(  U\otimes_{A}M,V\right)  \overset
{\pi^{\ast}}{\longrightarrow}\operatorname{Hom}_{\mathcal{M}}\left(  U\otimes
M,V\right)
\tikz[baseline=-.4ex]{ \draw[->] (0,0) -- node[above]{${\text{ }\eta_{U}\text{ }}$}(1cm,0);}\;\operatorname{Hom}%
_{\mathcal{C}}\left(  U,\underline{\operatorname{Hom}}\left(  M,V\right)
\right)  ,
\]
where $\pi:U\otimes M\rightarrow U\otimes_{A}M$ is the natural epimorphism in
$\mathcal{M}$. That is,%
\[
\operatorname{Im}\left(  \eta_{U}\,\pi^{\ast}\right)  =\operatorname{Hom}%
_{A}\left(  U,\underline{\operatorname{Hom}}\left(  M,V\right)  \right)
\subseteq\operatorname{Hom}_{\mathcal{C}}\left(  U,\underline
{\operatorname{Hom}}\left(  M,V\right)  \right)  .
\]
Let $f\in\operatorname{Hom}_{\mathcal{M}}\left(  U\otimes M,V\right)  $. We
have that $f\in\operatorname{Im}\pi^{\ast}$ if and only if the diagram
\begin{equation}%
\def\mleftdelim{.}\def\mrightdelim{.}\def\mrowsep{1cm}\def\mcolumnsep{2cm}%
\begin{tikzpicture}[scale=1,samples=100,baseline] \matrix (m) [matrix of math nodes,left delimiter={\mleftdelim},right delimiter={\mrightdelim},row sep=\mrowsep,column sep=\mcolumnsep]{ \left( U \otimes A\right) \otimes M & U \otimes M\\ U \otimes\left( A\otimes M\right) & \\ U \otimes M & V\\};\begin{scope}[every node/.style={midway,auto,font=\scriptsize}] \draw[->] (m-1-1) -- node[above] {$q\otimes id$} (m-1-2); \draw[->] (m-1-1) -- node{$a_{U,A,M}$} (m-2-1); \draw[->] (m-2-1) -- node{$id\otimes ev_{M,M}$} (m-3-1); \draw[->] (m-1-2) -- node[left] {$f$}(m-3-2); \draw[->] (m-3-1) -- node[above] {$f$}(m-3-2); \end{scope} \end{tikzpicture}
\end{equation}
commutes, that is, $f\circ\left(  q\otimes id\right)  =f\circ\left(  id\otimes
ev_{M,M}\right)  \circ a_{U,A,M}$. On the other hand, we have that $\eta
_{U}\left(  f\right)  \in\operatorname{Hom}_{A}\left(  U,\underline
{\operatorname{Hom}}\left(  M,V\right)  \right)  $ if and only if diagram
\begin{equation}%
\def\mleftdelim{.}\def\mrightdelim{.}\def\mrowsep{1cm}\def\mcolumnsep{2cm}%
\begin{tikzpicture}[scale=1,samples=100,baseline]\matrix (m) [matrix of math nodes,left delimiter={\mleftdelim},right delimiter={\mrightdelim},row sep=\mrowsep,column sep=\mcolumnsep]{ U \otimes A & U \\ \underline{\operatorname{Hom}}\left( M,V \right)\otimes A & \underline {\operatorname{Hom}}\left( M,V\right) \\};\begin{scope}[every node/.style={midway,auto,font=\scriptsize}] \draw[->] (m-1-1) -- node[above] {$q$} (m-1-2); \draw[->] (m-1-1) -- node{$\eta_{U}\left(f\right)\otimes id$} (m-2-1); \draw[->] (m-1-2) -- node {$\eta_{U}\left(f\right)$}(m-2-2); \draw[->] (m-2-1) -- node[above] {$\mu_{M,M,V}$}(m-2-2); \end{scope} \end{tikzpicture}
\end{equation}
commutes, i.e. $\mu_{M,M,V}\circ\left(  \eta_{U}\left(  f\right)  \otimes
id\right)  =\eta_{U}\left(  f\right)  \circ q$. Consider the diagram
\begin{equation}%
\def\mleftdelim{.}\def\mrightdelim{.}\def\mrowsep{1cm}\def\mcolumnsep{4cm}%
\begin{tikzpicture}[scale=1,samples=100,baseline]\matrix (m) [matrix of math nodes,left delimiter={\mleftdelim},right delimiter={\mrightdelim},row sep=\mrowsep,column sep=\mcolumnsep]{ \left( U \otimes A\right) \otimes M & \left( \underline{\operatorname{Hom}}\left( M,V\right) \otimes A\right) \otimes M\\ U \otimes\left( A\otimes M\right) & \underline{\operatorname{Hom}}\left( M,V\right) \otimes\left( A\otimes M\right) \\ U \otimes M & \underline {\operatorname{Hom}}\left( M,V \right) \otimes M\\ & V\\};\begin{scope}[every node/.style={midway,auto,font=\scriptsize}] \draw[->] (m-1-1) -- node[above] {$\left( \eta_{U}\left( f\right) \otimes id\right) \otimes id$} node[below=.5cm]{(I)} (m-1-2); \draw[->] (m-1-1) -- node{$a_{U,A,M}$} (m-2-1); \draw[->] (m-1-2) -- node[left] {$a_{\underline{\operatorname{Hom}}\left( M,V\right) ,A,M}$} (m-2-2); \draw[->] (m-2-1) -- node[above] {$\eta_{U }\left( f\right) \otimes\left( id\otimes id\right) $} node[below=.5cm]{(II)} (m-2-2); \draw[->] (m-2-1) -- node{$id\otimes ev_{M,M}$} (m-3-1); \draw[->] (m-2-2) -- node[left] {$id\otimes ev_{M,M}$}(m-3-2); \draw[->] (m-3-1) -- node[above] {$\eta_{U }\left( f\right) \otimes id$} node[below=.25cm,near end]{(III)} (m-3-2); \draw[->] (m-3-1) -- node[below left] {$f$}(m-4-2); \draw[->] (m-3-2) -- node[left] {$ev_{M,V}$}(m-4-2); \end{scope} \end{tikzpicture}
\end{equation}
The functoriality of $a$ implies the commutativity of the upper rectangle (I).
Since $\otimes$ is a bifunctor, the middle rectangle (II) commutes. By the
definition of the evaluation morphism $ev_{M,V}$ and the functoriality of
$\eta,$ we obtain commutativity of the triangle (III). Hence, the outside
quadrangle commutes, i.e.
\[
ev_{M,V}\circ\left(  id\otimes ev_{M,M}\right)  \circ a_{\underline
{\operatorname{Hom}}\left(  M,V\right)  ,A,M}\circ\left(  \left(  \eta
_{U}\left(  f\right)  \otimes id\right)  \otimes id\right)  =f\circ\left(
id\otimes ev_{M,M}\right)  \circ a_{U,A,M}.
\]
By the definition of $\mu_{M,M,V}$,
\[
\eta_{\underline{\operatorname{Hom}}\left(  M,V\right)  \otimes A}^{-1}\left(
\mu_{M,M,V}\right)  =ev_{M,V}\circ\left(  id\otimes ev_{M,M}\right)  \circ
a_{\underline{\operatorname{Hom}}\left(  M,V\right)  ,A,M},
\]
then
\[
\eta_{\underline{\operatorname{Hom}}\left(  M,V\right)  \otimes A}^{-1}\left(
\mu_{M,M,V}\right)  \circ\left(  \left(  \eta_{U}\left(  f\right)  \otimes
id\right)  \otimes id\right)  =f\circ\left(  id\otimes ev_{M,M}\right)  \circ
a_{U,A,M}.
\]
Because of the functoriality of $\eta$, the diagrams
\begin{equation}%
\def\mleftdelim{.}\def\mrightdelim{.}\def\mrowsep{1cm}\def\mcolumnsep{2cm}%
\begin{tikzpicture}[scale=1,samples=100,baseline]\matrix (m) [matrix of math nodes,left delimiter={\mleftdelim},right delimiter={\mrightdelim},row sep=\mrowsep,column sep=\mcolumnsep]{ \operatorname{Hom}_{\mathcal{M}}\left( U \otimes M,V\right) & \operatorname{Hom}_{\mathcal{C}}\left( U ,\underline{\operatorname{Hom}}\left( M,V\right) \right) \\ \operatorname{Hom}_{\mathcal{M}}\left( U \otimes A\otimes M,V\right) & \operatorname{Hom}_{\mathcal{C}}\left( U \otimes A,\underline {\operatorname{Hom}}\left( M,V\right) \right) \\}; \begin{scope}[every node/.style={midway,auto,font=\scriptsize}] \draw[->] (m-1-1) -- node[above] {$\eta_{U}$} (m-1-2); \draw[->] (m-1-1) -- node{$\operatorname{Hom}_{\mathcal{M}}\left( q\otimes id,V\right) $} (m-2-1); \draw[->] (m-1-2) -- node{$\operatorname{Hom}_{\mathcal{C}}\left( q,\underline{\operatorname{Hom}}\left( M,V\right) \right) $}(m-2-2); \draw[->] (m-2-1) -- node[above] {$\eta_{U \otimes A}$}(m-2-2); \end{scope} \end{tikzpicture}
\end{equation}
and
\begin{equation}%
\def\mleftdelim{.}\def\mrightdelim{.}\def\mrowsep{1cm}\def\mcolumnsep{2cm}%
\begin{tikzpicture}[scale=1,samples=100,baseline]\matrix (m) [matrix of math nodes,left delimiter={\mleftdelim},right delimiter={\mrightdelim},row sep=\mrowsep,column sep=\mcolumnsep]{ \operatorname{Hom}_{\mathcal{C}}\left( \underline{\operatorname{Hom}}\left( M,V\right) \otimes A,\underline{\operatorname{Hom}}\left( M,V\right) \right) & \operatorname{Hom}_{\mathcal{M}}\left( \left( \underline{\operatorname{Hom}}\left( M,V\right) \otimes A\right) \otimes M,V\right) \\ \operatorname{Hom}_{\mathcal{C}}\left( U \otimes A,\underline{\operatorname{Hom}}\left( M,V\right) \right) & \operatorname{Hom}_{\mathcal{M}}\left( \left( U \otimes A\right) \otimes M,V\right) \\}; \begin{scope}[every node/.style={midway,auto,font=\scriptsize}] \draw[->] (m-1-1) -- node[above] {$\eta_{\underline{\operatorname{Hom}}\left( M,V\right) \otimes A}^{-1}$} (m-1-2); \draw[->] (m-1-1) -- node{$\operatorname{Hom}_{\mathcal{C}}\left( \eta_{U}\left( f\right) \otimes id,\underline{\operatorname{Hom}}\left( M,V\right) \right) $} (m-2-1) ; \draw[->] (m-1-2) -- node {$\operatorname{Hom}_{\mathcal{M}}\left( \left( \eta _{U}\left( f\right) \otimes id\right) \otimes id,V\right)$}(m-2-2); \draw[->] (m-2-1) -- node[below] {$\eta_{U \otimes A}^{-1}$}(m-2-2); \end{scope} \end{tikzpicture}
\end{equation}
commute. So we have
\[
\eta_{U\otimes A}^{-1}\left(  \eta_{U}\left(  f\right)  \circ q\right)
=f\circ\left(  q\otimes id\right)
\]
and
\begin{align*}
\eta_{U\otimes A}^{-1}\left(  \mu_{M,M,V}\circ\left(  \eta_{U}\left(
f\right)  \otimes id\right)  \right)   &  =\eta_{\underline{\operatorname{Hom}%
}\left(  M,V\right)  \otimes A}^{-1}\left(  \mu_{M,M,V}\right)  \circ\left(
\left(  \eta_{U}\left(  f\right)  \otimes id\right)  \otimes id\right) \\
&  =f\circ\left(  id\otimes ev_{M,M}\right)  \circ a_{U,A,M}.
\end{align*}
Thus $f\in\operatorname{Im}\pi^{\ast}$ if and only if $\eta_{U}\left(
f\right)  \in\operatorname{Hom}_{A}\left(  U,\underline{\operatorname{Hom}%
}\left(  M,V\right)  \right)  $. Then the isomorphism $\zeta_{U,V}%
:\operatorname{Hom}_{\mathcal{M}}\left(  U\otimes_{A}M,V\right)
\rightarrow\operatorname{Hom}_{A}\left(  U,\underline{\operatorname{Hom}%
}\left(  M,V\right)  \right)  $ follows, and the naturality of $\zeta$ in $U$
and $V$ is routine.
\end{proof}

As a summary of Proposition \ref{propiHom&otA} and Lemma \ref{lemOstrik}, we have:

\begin{cor}
\label{corHomiso}Let $\mathcal{M}$ be a semisimple module category over a
fusion category $\mathcal{C}$. Let $M$ be a generator of $\mathcal{M}$, and
$A=\underline{\operatorname{Hom}}\left(  M,M\right)  $. We have canonical
isomorphisms%
\[
\operatorname{Hom}_{\mathcal{M}}\left(  \underline{\operatorname{Hom}}\left(
M,V\right)  \otimes_{A}M,V^{\prime}\right)  \cong\operatorname{Hom}_{A}\left(
\underline{\operatorname{Hom}}\left(  M,V\right)  ,\underline
{\operatorname{Hom}}\left(  M,V^{\prime}\right)  \right)  \cong
\operatorname{Hom}_{\mathcal{M}}\left(  V,V^{\prime}\right)  ,
\]
for any $V,V^{\prime}\in\mathcal{M}$.
\end{cor}

In \cite[Theorem 3.1]{Ostrik2003module} and \cite[Theorem 7.10.1]%
{Etingof2015tensor}, all semisimple module categories $\mathcal{M}$ with a
generator $M$ over a fusion category $\mathcal{C}$ are characterized. The
result was listed in Lemma \ref{lemOstrik}. Let $A=\underline
{\operatorname{Hom}}\left(  M,M\right)  $, then the functor $F=\underline
{\operatorname{Hom}}\left(  M,\bullet\right)  :\mathcal{M}\rightarrow
\mathrm{Mod}_{\mathcal{C}}(A)$ gives an equivalence of $\mathcal{C}$-module categories.

With the help of Proposition~\ref{propiHom&otA}, we can present a
quasi-inverse of $F$.

\begin{thm}
\label{theorem-G-is-qinv}Let $\mathcal{M}$ be a $k$-linear abelian semisimple
module category over a fusion category $\mathcal{C}$. If $M$ is a generator of
$\mathcal{M}$, then $A=\underline{\operatorname{Hom}}\left(  M,M\right)  $ is
a semisimple algebra in $\mathcal{C}$, and the functor $G=\bullet\otimes
_{A}M:\mathrm{Mod}_{\mathcal{C}}(A)\rightarrow\mathcal{M}$ is a quasi-inverse
to the equivalence $F=\underline{\operatorname{Hom}}\left(  M,\bullet\right)
:\mathcal{M}\rightarrow\mathrm{Mod}_{\mathcal{C}}(A)$. Moreover, if
$V\in\mathcal{M}$ and $U\in\mathrm{Mod}_{\mathcal{C}}(A)$, then

\begin{enumerate}
\item $\underline{\operatorname{Hom}}\left(  M,V\right)  \otimes_{A}M\cong V$
as objects of $\mathcal{M}$, $U\cong{}\underline{\operatorname{Hom}}\left(
M,U\otimes_{A}M\right)  $ as right $A$-modules in $\mathcal{C}$.

\item $V$ is a simple object of $\mathcal{M}$ if and only if $\underline
{\operatorname{Hom}}\left(  M,V\right)  $ is a simple object of $\mathrm{Mod}%
_{\mathcal{C}}(A)$.

\item $U$ is a simple object of $\mathrm{Mod}_{\mathcal{C}}(A)$ if and only if
$U\otimes_{A}M$ is a simple object of $\mathcal{M}$.
\end{enumerate}
\end{thm}

\begin{proof}
Since $\mathcal{M}$ is semisimple, $\operatorname{Hom}_{\mathcal{M}}\left(
\underline{\operatorname{Hom}}\left(  M,V\right)  \otimes_{A}M,V^{\prime
}\right)  \cong\operatorname{Hom}_{\mathcal{M}}\left(  V,V^{\prime}\right)  $
for all $V^{\prime}\in\mathcal{M}$ implies that $\underline{\operatorname{Hom}%
}\left(  M,V\right)  \otimes_{A}M\cong V$. By Lemma \ref{lemOstrik}, $U$ is of
the form $\underline{\operatorname{Hom}}\left(  M,V^{\prime}\right)  $, for
some $V^{\prime}\in\mathcal{M}$. Then we have%
\[
U\cong\underline{\operatorname{Hom}}\left(  M,V^{\prime}\right)
\cong\underline{\operatorname{Hom}}\left(  M,\underline{\operatorname{Hom}%
}\left(  M,V^{\prime}\right)  \otimes_{A}M\right)  \cong\underline
{\operatorname{Hom}}\left(  M,U\otimes_{A}M\right)  .
\]
By Proposition \ref{propiHom&otA}, there exists a natural isomorphism
\[
\zeta_{U,V}:\operatorname{Hom}_{\mathcal{M}}\left(  U\otimes_{A}M,V\right)
\rightarrow\operatorname{Hom}_{A}\left(  U,\underline{\operatorname{Hom}%
}\left(  M,V\right)  \right)  .
\]
Let $\theta_{U}=\zeta_{U,U\otimes_{A}M}\left(  id_{U\otimes_{A}M}\right)
\in\operatorname{Hom}_{A}\left(  U,\underline{\operatorname{Hom}}\left(
M,U\otimes_{A}M\right)  \right)  $. We claim that $\theta:$ $id_{\mathrm{Mod}%
_{\mathcal{C}}(A)}\rightarrow FG$ is a natural isomorphism. It is clear that
$\theta_{U}$ is an isomorphism. It suffices to show that the diagram
\begin{equation}%
\def\mleftdelim{.}\def\mrightdelim{.}\def\mrowsep{1cm}\def\mcolumnsep{2cm}%
\begin{tikzpicture}[scale=1,samples=100,baseline]\matrix (m) [matrix of math nodes,left delimiter={\mleftdelim},right delimiter={\mrightdelim},row sep=\mrowsep,column sep=\mcolumnsep]{ U & \underline{\operatorname{Hom}}\left( M,U\otimes _{A}M\right) \\ U^{\prime}& \underline{\operatorname{Hom}}\left( M,U^{\prime}\otimes _{A}M\right) \\}; \begin{scope}[every node/.style={midway,auto,font=\scriptsize}] \draw[->] (m-1-1) -- node[above] {$\theta_U$} (m-1-2); \draw[->] (m-1-1) -- node{$f$} (m-2-1); \draw[->] (m-1-2) -- node {$\underline{\operatorname{Hom}}\left( M,f\otimes _{A}id\right) $}(m-2-2); \draw[->] (m-2-1) -- node[above] {$\theta_{U^\prime}$}(m-2-2); \end{scope} \end{tikzpicture} \label{Nat-thetaU}%
\end{equation}
commutes, for any morphism $f:U\rightarrow U^{\prime}$ in $\mathrm{Mod}%
_{\mathcal{C}}(A)$. Indeed, this follows from the commutativity of the
following two diagrams,%

\begin{equation}%
\def\mleftdelim{.}\def\mrightdelim{.}\def\mrowsep{1cm}\def\mcolumnsep{2cm}%
\begin{tikzpicture}[scale=1,samples=100,baseline]\matrix (m) [matrix of math nodes,left delimiter={\mleftdelim},right delimiter={\mrightdelim},row sep=\mrowsep,column sep=\mcolumnsep]{ \operatorname{Hom}_{\mathcal{M}}\left( U\otimes _{A}M,U\otimes _{A}M\right) &\operatorname{Hom}_{A}\left( U,\underline{\operatorname{Hom}}\left( M,U\otimes _{A}M\right) \right) \\ \operatorname{Hom}_{\mathcal{M}}\left( U\otimes _{A}M,U^{\prime }\otimes _{A}M\right)& \operatorname{Hom}_{A}\left( U,\underline{\operatorname{Hom}}\left( M,U^{\prime }\otimes _{A}M\right) \right) \\ }; \begin{scope}[every node/.style={midway,auto,font=\scriptsize}] \draw[->] (m-1-1) -- node[above] {$\zeta_{U,U\otimes_AM}$} (m-1-2); \draw[->] (m-1-1) -- node[left] {$\operatorname{Hom}_{\mathcal{M}}\left( U\otimes _{A}M,f\otimes _{A}id\right) $} (m-2-1); \draw[->] (m-1-2) -- node {$\operatorname{Hom}_{A}\left( U,\underline{\operatorname{Hom}}\left( M,f\otimes _{A}id\right) \right)$}(m-2-2); \draw[->] (m-2-1) -- node[above] {$\zeta_{U,U^{\prime}\otimes_AM}$}(m-2-2); \end{scope} \end{tikzpicture} \label{thetaU}%
\end{equation}%
\def\mleftdelim{.}\def\mrightdelim{.}\def\mrowsep{1cm}\def\mcolumnsep{2cm}%
\begin{equation}%
\def\mleftdelim{.}\def\mrightdelim{.}\def\mrowsep{1cm}\def\mcolumnsep{2cm}%
\begin{tikzpicture}[scale=1,samples=100,baseline]\matrix (m) [matrix of math nodes,left delimiter={\mleftdelim},right delimiter={\mrightdelim},row sep=\mrowsep,column sep=\mcolumnsep]{ \operatorname{Hom}_{\mathcal{M}}\left( U^{\prime }\otimes _{A}M,U^{\prime }\otimes _{A}M\right)& \operatorname{Hom}_{A}\left( U^{\prime },\underline{\operatorname{Hom}}\left( M,U^{\prime }\otimes _{A}M\right) \right)\\ \operatorname{Hom}_{\mathcal{M}}\left( U\otimes _{A}M,U^{\prime }\otimes _{A}M\right)& \operatorname{Hom}_{A}\left( U,\underline{\operatorname{Hom}}\left( M,U^{\prime }\otimes _{A}M\right) \right) \\ }; \begin{scope}[every node/.style={midway,auto,font=\scriptsize}] \draw[->] (m-1-1) -- node[above] {$\zeta_{U^{\prime},U^{\prime}\otimes_AM}$} (m-1-2); \draw[->] (m-2-1) -- node[above] {$\zeta_{U,U^{\prime}\otimes_AM}$}(m-2-2); \draw[->] (m-1-1) -- node[left] {$\operatorname{Hom}_{\mathcal{M}}\left( f\otimes _{A}id,U^{\prime }\otimes _{A}M\right)$}(m-2-1); \draw[->] (m-1-2) -- node[right] {$\operatorname{Hom}_{A}\left( f,\underline{\operatorname{Hom}}\left( M,U^{\prime }\otimes _{A}M\right) \right)$}(m-2-2); \end{scope} \end{tikzpicture} \label{thetaUP}%
\end{equation}
The down-right images of the identity maps in diagrams (\ref{thetaU}) and
(\ref{thetaUP}) are the same; the right-down image of the identity map in
diagram (\ref{thetaU}) is $\underline{\operatorname{Hom}}\left(
M,f\otimes_{A}id\right)  \circ\theta_{U}$, and the right-down image of the
identity map in diagram (\ref{thetaUP}) is $\theta_{U^{\prime}}\circ f$, thus
$\underline{\operatorname{Hom}}\left(  M,f\otimes_{A}id\right)  \circ
\theta_{U}=\theta_{U^{\prime}}\circ f$. This completes the commutativity of
the diagram (\ref{Nat-thetaU}). Hence $G$ is a quasi-inverse of $F$.
\end{proof}

Let $\left(  H,R\right)  $ be a quasi-triangular Hopf algebra, and let
$\mathcal{C}={}_{H^{cop}}\mathcal{M}$ be the category of left $H^{cop}%
$-modules. Then the category $_{H}^{H_{R}}\mathcal{M}$ has a structure of
module category over $\mathcal{C}$ by considering $X\otimes M$ as an object of
$_{H}^{H_{R}}\mathcal{M}$ via the $H$-action
\[
h\left(  x\otimes m\right)  =\sum h_{\left(  2\right)  }x\otimes h_{\left(
1\right)  }m,\quad\forall h\in H,x\in X,m\in M,
\]
and the left $H_{R}$-coaction on the right tensorand $M$, for any $X\in
{}\mathcal{C},M\in\mathcal{M}$.

Assume further that $H$ is semisimple and cosemisimple. Let%
\[
H=D_{1}\oplus\cdots\oplus D_{r}%
\]
be the decomposition in Proposition \ref{propHdecomp}, where $D_{1}%
,\ldots,D_{r}$ are irreducible Yetter-Drinfeld submodules of $H$, which are
also the minimal $H $-adjoint-stable subcoalgebras of $H_{R}$. For any $V\in
{}_{H}^{H_{R}}\mathcal{M}$, set $V_{i}=\left\{  v\in V\mid\rho_{R}\left(
v\right)  \in D_{i}\otimes V\right\}  ,$ $1\leq i\leq r.$ Clearly,
$V=\oplus_{i=1}^{r}V_{i}.$ Since $D_{i}\subseteq H$ is an $H$-submodule, by
(\ref{eqb}) we have $V_{i}\in{}_{H}^{D_{i}}\mathcal{M}$, the category of
$D_{i}$-comodule in $_{H}\mathcal{M}$. Thus the category ${}_{H}%
^{H}\mathcal{YD}={}_{H}^{H_{R}}\mathcal{M}$ is a direct sum of $\mathcal{C}%
$-module subcategories%
\begin{equation}
_{H}^{H_{R}}\mathcal{M=}{}_{H}^{D_{1}}\mathcal{M}\oplus\cdots\oplus{}%
_{H}^{D_{r}}\mathcal{M}. \label{decomp HRHM}%
\end{equation}
Moreover, we will see that each direct summand is an indecomposable
$\mathcal{C}$-module category.

\begin{lem}
\label{lemDHMirr}Let $D$ be a minimal $H$-adjoint-stable subcoalgebra of
$H_{R}.$ Then $_{H}^{D}\mathcal{M}$ is an indecomposable module category over
$\mathcal{C}={}_{H^{cop}}\mathcal{M}$.
\end{lem}

\begin{proof}
The proof is similar to that of \cite[Proposition 1.18]%
{andruskiewitsch2007module}. Assume that $_{H}^{D}\mathcal{M}=\mathcal{M}%
_{1}\oplus\mathcal{M}_{2},$ where $\mathcal{M}_{1},\mathcal{M}_{2}$ are
$_{H^{cop}}\mathcal{M}$-module subcategories of $_{H}^{D}\mathcal{M}.$ If
$M\in{}_{H}^{D}\mathcal{M}$ then there exist $M_{1}\in\mathcal{M}_{1},M_{2}%
\in\mathcal{M}_{2}$ such that $M=M_{1}\oplus M_{2}.$ If $\alpha:M\rightarrow
N$ is a morphism in $_{H}^{D}\mathcal{M}$ then $\alpha\left(  M_{1}\right)
\subseteq N_{1},\alpha\left(  M_{2}\right)  \subseteq N_{2}.$

Then $D=D_{1}\oplus D_{2},$ for some $D_{1}\in\mathcal{M}_{1}$, $D_{2}%
\in\mathcal{M}_{2}.$ Since $D$ is irreducible, we may assume that $D=D_{1}%
\in\mathcal{M}_{1}$. For arbitrary $M\in{}_{H}^{D}\mathcal{M}$, then
$M=M_{1}\oplus M_{2}$ for some $M_{1}\in\mathcal{M}_{1},M_{2}\in
\mathcal{M}_{2}$. Note that $M\in\mathcal{M}\subseteq\mathcal{C}$,
$D\in\mathcal{M}$, then $M\otimes D\in\mathcal{M}$ by the $\mathcal{C}$-module
category structure on $\mathcal{M}$. Obviously, the map $\alpha_{M}%
:M\rightarrow M\otimes D,$ $m\mapsto\sum m^{\left\langle 0\right\rangle
}\otimes m^{\left\langle -1\right\rangle }$ is a morphism in $_{H}%
^{D}\mathcal{M}$, so $\alpha_{M}\left(  M_{1}\right)  \in\mathcal{M}_{1}$,
$\alpha_{M}\left(  M_{2}\right)  \in\mathcal{M}_{2}$. Since $M\in\mathcal{C}$,
we have $M\otimes D$ is an object in $\mathcal{M}_{1}$. Then $\alpha
_{M}\left(  M\right)  \subseteq M\otimes D\in\mathcal{M}_{1}$, hence
$\alpha_{M}\left(  M_{2}\right)  =0$ and thus $M_{2}=0.$ Therefore
${}\mathcal{M}_{2}$ is trivial, and $_{H}^{D}\mathcal{M}$ is indecomposable.
\end{proof}

\section{Irreducible Yetter-Drinfeld Modules}

\label{section-YDmodule}In this section, we assume that $(H,R)$ is a
semisimple and cosemisimple quasi-triangular Hopf algebra over $k.$ We will
classify all the irreducible objects of ${}_{H}^{H}\mathcal{YD}.$

Let $W$ be a left $H_{R}$-comodule, then $H\otimes W$ is also a left $H_{R}%
$-comodule via
\begin{equation}
{\rho_{R}}(h\otimes w)=\sum h_{(1)}\cdot_{ad}w^{\langle-1\rangle}\otimes
h_{(2)}\otimes w^{\langle0\rangle}, \label{eqe}%
\end{equation}
for all $h\in H$, $w\in W$. Then (\ref{eqc}) implies that $\rho_{R}$ is a
coaction. Furthermore, $H\otimes W$ is an object of $_{H}^{H_{R}}\mathcal{M}$
with left $H$-action given by the multiplication of $H$ on the left tensorand
$H$.

Let $D$ be a minimal $H$-adjoint-stable subcoalgebra of $H_{R}$. For a left
$D$-comodule $W$, $H\otimes W$ is in the subcategory $_{H}^{D}\mathcal{M}$ of
${}_{H}^{H_{R}}\mathcal{M}$. We will use the object $H\otimes W$ to
characterize the category $_{H}^{D}\mathcal{M}$.~

\subsection{Structure of $_{H}^{D}\mathcal{M}$ as a module category over
$\mathcal{C}={}_{H^{cop}}\mathcal{M}$}

By Lemma \ref{lemDHMirr}, $\mathcal{M}={}_{H}^{D}\mathcal{M}$ is an
indecomposable module subcategory of $_{H}^{H_{R}}\mathcal{M}$ over
$\mathcal{C={}}_{H^{cop}}\mathcal{M}$. We apply Theorem
\ref{theorem-G-is-qinv} to the module category $_{H}^{D}\mathcal{M}$ in this subsection.

For two objects $M_{1},M_{2}\in{}_{H}^{D}\mathcal{M}$, $\operatorname{Hom}%
^{D}\left(  M_{1},M_{2}\right)  $ is obviously an object of $_{H^{cop}%
}\mathcal{M}$, with the left $H$-action determined by
\[
\left(  h\cdot f\right)  \left(  m_{1}\right)  =\sum h_{\left(  2\right)
}f\left(  \left(  \bar{S}h_{\left(  1\right)  }\right)  m_{1}\right)  ,
\]
where $h\in H$, $f\in\operatorname{Hom}^{D}\left(  M_{1},M_{2}\right)  $,
$m_{1}\in M_{1}$.

For $X\in{}_{H^{cop}}\mathcal{M}$, the restriction of the canonical
isomorphism
\[
\operatorname{Hom}\left(  X,\operatorname{Hom}\left(  M_{1},M_{2}\right)
\right)  \cong\operatorname{Hom}\left(  X\otimes M_{1},M_{2}\right)
\]
on $\operatorname{Hom}_{H}\left(  X,\operatorname{Hom}^{D}\left(  M_{1}%
,M_{2}\right)  \right)  $ induces a natural isomorphism
\[
\operatorname{Hom}_{H}\left(  X,\operatorname{Hom}^{D}\left(  M_{1}%
,M_{2}\right)  \right)  \cong\operatorname{Hom}_{H}^{D}\left(  X\otimes
M_{1},M_{2}\right)  ,
\]
so the functor $X\mapsto$ $\operatorname{Hom}_{H}^{D}\left(  X\otimes
M_{1},M_{2}\right)  $ is representable by $\operatorname{Hom}^{D}\left(
M_{1},M_{2}\right)  $. Thus the internal Hom $\underline{\operatorname{Hom}%
}\left(  M_{1},M_{2}\right)  =\operatorname{Hom}^{D}\left(  M_{1}%
,M_{2}\right)  $. It's not difficult to verify that the evaluation map
$ev_{M_{1},M_{2}}$ in the $\mathcal{C}$-module category ${}_{H}^{D}%
\mathcal{M}$ is indeed the regular evaluation map.

\begin{prop}
\label{prop-C=HM-equiva}Let $D$ be a minimal $H$-adjoint-stable subcoalgebra
of $H_{R}$. For any nonzero left $D$-comodule $W,$ the algebra
$A=\operatorname{Hom}^{D}\left(  H\otimes W,H\otimes W\right)  $ in
$\mathcal{C}$ is semisimple, and the functors
\begin{align*}
F=\operatorname{Hom}^{D}\left(  H\otimes W,\bullet\right)   &  :{}_{H}%
^{D}\mathcal{M}\rightarrow\mathrm{Mod}_{\mathcal{C}}(A),\\
G=\bullet\otimes_{A}\left(  H\otimes W\right)   &  :\mathrm{Mod}_{\mathcal{C}%
}(A)\rightarrow{}_{H}^{D}\mathcal{M}%
\end{align*}
establish an equivalence of $\mathcal{C}$-module categories between $_{H}%
^{D}\mathcal{M}$ and $\mathrm{Mod}_{\mathcal{C}}(A)\mathrm{.}$
\end{prop}

\begin{proof}
Since $\mathcal{M}$ is indecomposable as a $\mathcal{C}$-module category, the
object $M=$ $H\otimes W$ generates $\mathcal{M}$. The result follows from
Theorem \ref{theorem-G-is-qinv}.
\end{proof}

\subsection{Structure of $_{H}^{D}\mathcal{M}$ as a module category over
$\mathcal{C}={}\mathrm{Vec}_{k}$}

In this subsection, we present another characterization of the category
$\mathcal{M}={}_{H}^{D}\mathcal{M}$. Let $\mathcal{C}=\mathrm{Vec}_{k}$ be the
category of finite dimensional $k$-vector spaces. The category ${}_{H}%
^{D}\mathcal{M}$ has a natural structure of module category over $\mathcal{C}$
by considering $X\otimes M$ as an object of $_{H}^{D}\mathcal{M}$ via the
$H$-action and $D$-coaction on the right tensorand $M$, for any $X\in
\mathrm{Vec}_{k},M\in{}_{H}^{D}\mathcal{M}$.

If $M_{1},M_{2}\in{}_{H}^{D}\mathcal{M}$, then there is a canonical natural
isomorphism
\[
\operatorname{Hom}_{H}^{D}\left(  \bullet\otimes M_{1},M_{2}\right)
\cong\operatorname{Hom}\left(  \bullet,\operatorname{Hom}_{H}^{D}\left(
M_{1},M_{2}\right)  \right)  .
\]
Thus the internal Hom $\underline{\operatorname{Hom}}\left(  M_{1}%
,M_{2}\right)  =\operatorname{Hom}_{H}^{D}\left(  M_{1},M_{2}\right)  $, and
the evaluation map $ev_{M_{1},M_{2}}$ is exactly the regular evaluation map.
Applying Theorem \ref{theorem-G-is-qinv} to $\mathcal{M}$, we get:

\begin{prop}
\label{prop-C=Vec-equiva}Let $D$ be a minimal $H$-adjoint-stable subcoalgebra
of $H_{R}$. If $W$ is a finite dimensional nonzero left $D$-comodule, then
$A=\operatorname{Hom}_{H}^{D}\left(  H\otimes W,H\otimes W\right)  $ is a
semisimple $k$-algebra, and the functors
\begin{align*}
F=\operatorname{Hom}_{H}^{D}\left(  H\otimes W,\bullet\right)   &  :{}_{H}%
^{D}\mathcal{M}\rightarrow\mathcal{M}_{A},\\
G=\bullet\otimes_{A}\left(  H\otimes W\right)   &  :\mathcal{M}_{A}%
\rightarrow{}_{H}^{D}\mathcal{M}%
\end{align*}
establish an equivalence of $\mathcal{C}$-module categories between $_{H}%
^{D}\mathcal{M}$ and $\mathcal{M}_{A}\mathrm{.}$
\end{prop}

\begin{proof}
To apply Theorem \ref{theorem-G-is-qinv} to the $\mathcal{C}$-module category
$\mathcal{M}={}_{H}^{D}\mathcal{M}$ and the object $M=H\otimes W\in
\mathcal{M}$, we only need to verify that $H\otimes W$ generates the
$\mathcal{C}$-module category $\mathcal{M}$. For any simple object
$V\in\mathcal{M}$, we claim that the internal Hom $\underline
{\operatorname{Hom}}\left(  M,V\right)  =\operatorname{Hom}_{H}^{D}\left(
M,V\right)  $ is nonzero. Let $D^{\prime}=\mathrm{span}\left\{  v^{\ast
}\rightharpoonup v\mid v\in V,v^{\ast}\in V^{\ast}\right\}  $ with $v^{\ast
}\rightharpoonup v=\sum v^{\left\langle -1\right\rangle }\left\langle v^{\ast
},v^{\left\langle 0\right\rangle }\right\rangle $. It is easy to check that
$D^{\prime}$ is a nonzero left coideal of $D$ and is also an $H$-submodule
under the left $H$-action $\cdot_{ad}$. Since $D$ is irreducible in ${}%
_{H}^{H}\mathcal{YD}$, then $D^{\prime}=D$. So there exists a surjection
$V^{\left(  n\right)  }\rightarrow D\rightarrow0$ in $^{D}\mathcal{M}$ for
some $n\in\mathbb{N}^{+}$. Since $^{D}\mathcal{M}$ is semisimple, there exists
an injection $0\rightarrow D\rightarrow$ $V^{\left(  n\right)  }$ in
$^{D}\mathcal{M}$. Take a simple $D$-subcomodule $W^{\prime}$ of $W$. Then
$W^{\prime}$ is isomorphic to a simple left coideal of $D$ and there exists a
left $D$-comodule injection $j:W^{\prime}\rightarrow V$. Thus the map
$M=H\otimes W\rightarrow V$ given by $h\otimes w\rightarrow hj\left(  p\left(
w\right)  \right)  $, $h\in H,w\in W$ is a nonzero morphism in $_{H}%
^{D}\mathcal{M}$, where $p:W\rightarrow W^{\prime}$ is a left $D$-comodule
projection. Then by the isomorphism $\operatorname{Hom}_{H}^{D}\left(
\operatorname{Hom}_{H}^{D}\left(  M,V\right)  \otimes M,V\right)
\cong\operatorname{Hom}\left(  \operatorname{Hom}_{H}^{D}\left(  M,V\right)
,\operatorname{Hom}_{H}^{D}\left(  M,V\right)  \right)  \neq0,$ so the
evaluation morphism $\operatorname{Hom}_{H}^{D}\left(  M,V\right)  \otimes
M\rightarrow V$ is a surjection in $\mathcal{M}$. Hence $M$ is a generator,
and the result follows.
\end{proof}

Take a finite dimensional left $D$-comodule $W$. Certainly, $W^{\ast}$ can be
considered as a natural right $D$-comodule. To be specific, $\sum\left\langle
{{w}^{\ast\left\langle 0\right\rangle }},w\right\rangle {{w}^{\ast\left\langle
1\right\rangle }}=\sum{w}^{\left\langle -1\right\rangle }\left\langle w^{\ast
},{w}^{\left\langle 0\right\rangle }\right\rangle $, for all $w\in W$,
$w^{\ast}\in W^{\ast}$.

\begin{prop}
\label{propiHomHotW}Let $W$ be a finite dimensional nonzero left $D$-comodule.
Then there is a functor equivalence $\operatorname{Hom}_{H}^{D}\left(
H\otimes W,\bullet\right)  \simeq W^{\ast}\square_{D}\bullet$, where
$\square_{D}$ denotes the cotensor product for $D$-comodules.
\end{prop}

\begin{proof}
Let $\left\{  w_{i},w_{i}^{\ast}\right\}  _{i=1}^{s}$ be a dual basis for $W.$
For any $V\in{}_{H}^{D}\mathcal{M},$ we define a linear map
\[
\mu_{V}:\operatorname{Hom}_{H}^{D}\left(  H\otimes W,V\right)  \rightarrow
W^{\ast}\square_{D}V
\]
given by%
\[
f\mapsto\sum_{i=1}^{s}w_{i}^{\ast}\otimes f\left(  1\otimes w_{i}\right)  .
\]
Since $f\in\operatorname{Hom}_{H}^{D}\left(  H\otimes W,V\right)  $ is a left
$D$-comodule map, $\sum_{i=1}^{s}w_{i}^{\ast}\otimes f\left(  1\otimes
w_{i}\right)  \in W^{\ast}\square_{D}V,$ and the map $\mu_{V}$ is
well-defined. Conversely, define a map
\[
\nu_{V}:W^{\ast}\square_{D}V\rightarrow\operatorname{Hom}_{H}^{D}\left(
H\otimes W,V\right)
\]
given by
\[
\nu_{V}\left(  \sum_{i=1}^{s}w_{i}^{\ast}\otimes v_{i}\right)  \left(
h\otimes w\right)  =h\sum_{i=1}^{s}\left\langle w_{i}^{\ast},w\right\rangle
v_{i},
\]
for any $\sum_{i=1}^{s}w_{i}^{\ast}\otimes v_{i}\in W^{\ast}\square_{D}V,h\in
H,w\in W.$ Direct verification shows that $\nu_{V}\left(  \sum_{i=1}^{s}%
w_{i}^{\ast}\otimes v_{i}\right)  $ is a morphism in $_{H}^{D}\mathcal{M}$,
then $\nu_{V}$ is well defined. It is easy to check that $\nu_{V}$ and
$\mu_{V}$ are mutually inverse maps. For any $V_{1},V_{2}\in{}_{H}%
^{D}\mathcal{M}$ and any $\alpha\in\operatorname{Hom}_{H}^{D}\left(
V_{1},V_{2}\right)  ,$%
\begin{align*}
\left(  id\otimes\alpha\right)  \circ\mu_{V_{1}}\left(  f\right)   &
=\sum_{i=1}^{s}w_{i}^{\ast}\otimes\left(  \alpha\circ f\right)  \left(
1\otimes w_{i}\right) \\
&  =\mu_{V_{2}}\left(  \alpha\circ f\right)  ,
\end{align*}
for all $f\in\operatorname{Hom}_{H}^{D}\left(  H\otimes W,V_{1}\right)  .$ It
follows that $\mu=\left\{  \mu_{V}\mid V\in{}_{H}^{D}\mathcal{M}\right\}  $ is
a natural isomorphism of the functor $\operatorname{Hom}_{H}^{D}\left(
H\otimes W,\bullet\right)  $ to $W^{\ast}\square_{D}\bullet$. Therefore, we
have functor equivalences $\underline{\operatorname{Hom}}\left(  H\otimes
W,\bullet\right)  \simeq\operatorname{Hom}_{H}^{D}\left(  H\otimes
W,\bullet\right)  \simeq W^{\ast}\square_{D}\bullet.$
\end{proof}

We now state another characterization of the algebra
\[
A=\underline{\operatorname{Hom}}\left(  H\otimes W,H\otimes W\right)
=\operatorname{Hom}_{H}^{D}\left(  H\otimes W,H\otimes W\right)
\]
in Proposition \ref{prop-C=Vec-equiva}, if $W$ is finite dimensional.

For any two finite dimensional left $D$-comodules $W$ and $W^{\prime}$, we
write
\[
N_{WW^{\prime}}=W^{\ast}\square_{D}\left(  H\otimes W^{\prime}\right)
\cong\underline{\operatorname{Hom}}\left(  H\otimes W,H\otimes W^{\prime
}\right)  .
\]
Let $W^{\prime\prime}$ be another finite dimensional left $D$-comodule.
Deduced from the composition of the internal Hom, we get
\[
x\circ y=\sum_{i=1}^{n}\sum_{j=1}^{m}w_{i}^{\ast}\otimes h_{i}g_{j}%
\otimes\left\langle w_{j}^{\prime\ast},w_{i}^{\prime}\right\rangle
w_{j}^{\prime\prime},
\]
for any $x=\sum_{j=1}^{m}w_{j}^{\prime\ast}\otimes g_{j}\otimes w_{j}%
^{\prime\prime}\in N_{W^{\prime}W^{\prime\prime}}$, $y=\sum_{i=1}^{n}%
w_{i}^{\ast}\otimes h_{i}\otimes w_{i}^{\prime}\in N_{WW^{\prime}}$. If
$W=W^{\prime}$, we write $N_{W}=N_{WW}$. Then $\left(  N_{W},\circ\right)  $
is an associative algebra with the identity element $\sum_{i=1}^{s}w_{i}%
^{\ast}\otimes1_{H}\otimes w_{i}\in N_{W}$, where $\left\{  w_{i},w_{i}^{\ast
}\right\}  _{i=1}^{s}$ is a dual basis for $W$. Furthermore, each
$N_{WW^{\prime}}$ is both an $N_{W^{\prime}}$-$N_{W}$-bimodule and a left
$H$-comodule via the coaction given by%
\[
\rho\left(  x\right)  =\sum_{i=1}^{n}S\left(  h_{i\left(  2\right)  }\right)
\otimes w_{i}^{\ast}\otimes h_{i\left(  1\right)  }\otimes w_{i}^{\prime
},\quad\text{for }x=\sum_{i=1}^{n}w_{i}^{\ast}\otimes h_{i}\otimes
w_{i}^{\prime}\in N_{WW^{\prime}}.
\]
Thus, $N_{W}$ is a left $H$-comodule algebra. We call it the $R$%
\emph{-adjoint-stable algebra} of $W$.

\begin{cor}
Let $W$ be a finite dimensional nonzero left $D$-comodule, then
\[
N_{W}=W^{\ast}\square_{D}\left(  H\otimes W\right)  \cong\operatorname{Hom}%
_{H}^{D}\left(  H\otimes W,H\otimes W\right)
\]
as algebras.
\end{cor}

\begin{rem}
For finite dimensional left $D$-comodule $W$, the algebra $\operatorname{Hom}%
^{D}\left(  H\otimes W,H\otimes W\right)  $ in Proposition
\ref{prop-C=HM-equiva} is indeed isomorphic to the smash product
$N_{W}\#H^{\ast op}$.
\end{rem}

By the isomorphism $N_{W}\cong\operatorname{Hom}_{H}^{D}\left(  H\otimes
W,H\otimes W\right)  $, the evaluation map in $\mathcal{C}$-module category
${}_{H}^{D}\mathcal{M}$ gives a natural left $N_{W}$-action on $H\otimes W$
with
\[
\left(  \sum_{i=1}^{n}w_{i}^{\ast}\otimes h_{i}\otimes w_{i}\right)
\cdot\left(  h\otimes w\right)  =\sum_{i=1}^{n}hh_{i}\otimes\left\langle
w_{i}^{\ast},w\right\rangle w_{i},
\]
where $h\in H,w\in W,\sum_{i=1}^{n}w_{i}^{\ast}\otimes h_{i}\otimes w_{i}\in
N_{W}.$

As a consequence of Proposition \ref{prop-C=Vec-equiva}, we have

\begin{thm}
\label{thmDHM=NWM}Let $\left(  H,R\right)  $ be a semisimple and cosemisimple
quasi-triangular Hopf algebra, and let $D$ be a minimal $H$-adjoint-stable
subcoalgebra of $H_{R}$. If $W$ is a finite dimensional nonzero left
$D$-comodule, then $N_{W}$ is semisimple, and the functors
\begin{align*}
W^{\ast}\square_{D}\bullet{}  &  :{}_{H}^{D}\mathcal{M}\rightarrow
{}\mathcal{M}_{N_{W}},\\
\bullet\otimes_{N_{W}}\left(  H\otimes W\right)  {}  &  :\mathcal{M}_{N_{W}%
}\rightarrow{}_{H}^{D}\mathcal{M}%
\end{align*}
establish an equivalence between $_{H}^{D}\mathcal{M}$ and $\mathcal{M}%
_{N_{W}}$ as module categories over $\mathcal{C}=$ $\mathrm{Vec}_{k}$.
\end{thm}

Let%
\[
H_{R}=D_{1}\oplus\cdots\oplus D_{r}%
\]
be the decomposition in Proposition \ref{propHdecomp}, where $D_{1}%
,\ldots,D_{r}$ are minimal $H$-adjoint-stable subcoalgebras of $H_{R}$. Let
$W$ be a simple left coideal of $\left(  H_{R},\Delta_{R}\right)  $, then
$W\subseteq D_{i}$, for some $1\leq i\leq r$. By (\ref{eqc})
\[
\Delta_{R}\left(  H\cdot_{ad}W\right)  \subseteq H\cdot_{ad}D_{i}\otimes
H\cdot_{ad}W\subseteq D_{i}\otimes H\cdot_{ad}W,
\]
which implies $H\cdot_{ad}W$ is an $H$-adjoint-stable subcoalgebra of $H_{R}$.
Note that $D_{i}$ is minimal, thus $D_{i}=H\cdot_{ad}W$. We also denote
$N_{W}=W^{\ast}\square_{D_{i}}\left(  H\otimes W\right)  $ the $R$%
-adjoint-stable algebra of $W$.

On the other hand, assume that $D$ is a minimal $H$-adjoint-stable
subcoalgebra of $H_{R}$. Then for any simple left coideal $W\ $of $D$, we have
$D=H\cdot_{ad}W$.

For simple left coideals $W$,$W^{\prime}$ of $\left(  H_{R},\Delta_{R}\right)
$, we say that $W$ and $W^{\prime}$ are conjugate, if
\begin{equation}
H\cdot_{ad}W=H\cdot_{ad}W^{\prime},\text{ denoted by }W\sim W^{\prime}.
\label{eq_conjugate}%
\end{equation}
Then \textquotedblleft$\sim$\textquotedblright\ defines an equivalence
relation on the set of simple left coideals of $\left(  H_{R},\Delta
_{R}\right)  $.

Let $\left\{  W_{1},\ldots,W_{r}\right\}  $ be a complete set of
representatives for the equivalence classes. Then the decomposition in
Proposition \ref{propHdecomp} is
\begin{equation}
H=\left(  H\cdot_{ad}W_{1}\right)  \oplus\cdots\oplus\left(  H\cdot_{ad}%
W_{r}\right)  . \label{eqHDecom}%
\end{equation}
An irreducible Yetter-Drinfeld submodule of $H$ is called a conjugacy class of
$H.$ This notion of the conjugacy class for semisimple Hopf algebras was
introduced by Cohen and Westreich via primitive central idempotents of
character algebra \cite{cohen2010higman,cohen2010structure,cohen2011conjugacy}%
. For quasi-triangular Hopf algebra, it's actually a minimal $H$%
-adjoint-stable subcoalgebra of $H_{R}$.

\begin{thm}
\label{thmirr_mod_stru}Let $\Omega=\left\{  W_{1},\ldots,W_{r}\right\}  $ be a
complete set of representatives for the equivalence classes determined by
``$\sim$'' in (\ref{eq_conjugate}).

\begin{enumerate}
\item Then for any $W\in\Omega$ and any irreducible right $N_{W}$-module $U$,
$U$ $\otimes_{N_{W}}\left(  H\otimes W\right)  $ is irreducible in ${}_{H}%
^{H}\mathcal{YD}$.

\item Any irreducible module $V\in{}_{H}^{H}\mathcal{YD}$ is isomorphic to
$U\otimes_{N_{W}}\left(  H\otimes W\right)  $, for some $W\in\Omega$ and an
irreducible right $N_{W}$-module $U$.

\item As Yetter-Drinfeld modules, $U\otimes_{N_{W}}\left(  H\otimes W\right)
$ and $U^{\prime}$ $\otimes_{N_{W^{\prime}}}\left(  H\otimes W^{\prime
}\right)  $ are isomorphic if and only if $W=W^{\prime}\in\Omega$ and $U\cong
U^{\prime}$ as $N_{W}$-modules.
\end{enumerate}
\end{thm}

Let $D$ be a minimal $H$-adjoint-stable subcoalgebra of $H_{R}$. Assume that
$D$ contains a grouplike element $g$. We now apply our main results to this case.

Let $W=kg$, then%
\[
N_{W}=\left\{  h\in H\;\mid\sum h_{\left(  1\right)  }\cdot_{ad}g\otimes
h_{\left(  2\right)  }=g\otimes h\right\}
\]
is a right coideal subalgebra of $H.$ For a right $N_{W}$-module $U$, the
structure on $U\otimes_{N_{W}}H$ given by%
\[
h^{\prime}\cdot\left(  u\otimes h\right)  =u\otimes h^{\prime}h,
\]%
\[
\rho_{R}\left(  u\otimes h\right)  =\sum h_{\left(  1\right)  }\cdot
_{ad}g\otimes u\otimes h_{\left(  2\right)  },
\]
which makes $U$ $\otimes_{N_{W}}H$ into an object of ${}_{H}^{D}\mathcal{M}$.
By Theorem~\ref{thmirr_mod_stru}, we have

\begin{prop}
\label{PropifC=kg}For any irreducible right $N_{W}$-module $U,$ the object
$U\otimes_{N_{_{W}}}H\in{}_{H}^{D}\mathcal{M}$ is irreducible. Conversely, any
irreducible object of $_{H}^{D}\mathcal{M}$ is isomorphic to $U\otimes_{N_{W}%
}H$, for some irreducible right $N_{W}$-module $U$.
\end{prop}

The structure theorem of Yetter-Drinfeld module over a finite group algebra
follows immediately from Proposition \ref{PropifC=kg}.

\begin{cor}
[cf. \cite{dijkgraaf1992quasi,gould1993quantum,Andruskiewitsch1998Braided}%
]\label{corDGmodstru}Let $G$ be a finite group. For $g\in G$, denote by
$C\left(  g\right)  =\left\{  h\in G\mid hgh^{-1}=g\right\}  $ the centralizer
subgroup of $g$ and by $\mathcal{C}_{g}$ the conjugacy classes of $g$. Let $U$
be an irreducible left $kC\left(  g\right)  $-module. Then we have induced
$kG$-module $M\left(  g,U\right)  =kG\otimes_{kC\left(  g\right)  }U$. The
comodule structure on $M\left(  g,U\right)  $ is given by%
\[
\rho\left(  x\otimes u\right)  =xgx^{-1}\otimes x\otimes u,\forall x\in G,u\in
U,
\]
which makes $M\left(  g,U\right)  $ into an irreducible object of $_{kG}%
^{kG}\mathcal{YD}.$ Moreover, any irreducible object of $_{kG}^{kG}%
\mathcal{YD}$ is isomorphic to some $M\left(  g,U\right)  $.
\end{cor}

\begin{proof}
As algebras, $kC\left(  g\right)  \cong\left(  N_{kg}\right)  ^{op}$, then the
result follows.
\end{proof}

Now we give the structure of 1-dimensional Yetter-Drinfeld modules for a
quasi-triangular Hopf algebra $(H,R)$. Let $ZG(H)$ be the set of central
group-likes of $H$.

\begin{cor}
\label{cor_1-dim} Let $\left(  H,R\right)  $ be a quasi-triangular Hopf
algebra, then there is a 1-1 correspondence between the set of the
non-isomorphic 1-dimensional Yetter-Drinfeld modules in ${}_{H}^{H}%
\mathcal{YD}$ and the set
\[
ZG(H)\times\left\{  \text{the non-isomorphic 1-dimensional modules of
}H\right\}  .
\]

\end{cor}

\begin{proof}
Let $V=kv$ be a 1-dimensional Yetter-Drinfeld module in ${}_{H}^{H}%
\mathcal{YD}$. As a left $H_{R}$-comodule, $\rho_{R}\left(  v\right)
=g\otimes v,$ for a grouplike element $g\in G\left(  H_{R}\right)  $. Then
$D=kg$ is an $H$-adjoint-stable subcoalgebra of $H_{R}$, hence $g\in Z\left(
H\right)  .$ Moreover, $\rho\left(  g\right)  =\sum gR^{2}\otimes R^{1}%
\cdot_{ad}g=g\otimes g,$ i.e. $g\in G\left(  H\right)  .$ Conversely, for an
element $g\in$ $ZG(H),$ $D=kg$ is a Yetter-Drinfeld submodule of $H,$ and
$N_{kg}=H.$ For any $1$-dimensional $H$-module $V$, $V$ $\otimes_{N_{kg}}H$ is
a 1-dimensional Yetter-Drinfeld module. Explicitly, $\rho_{R}\left(
v\otimes1_{H}\right)  =g\otimes v\otimes1_{H}$, for $v\in V$.
\end{proof}

\begin{rem}
Similar statements in this paper can be proved for Yetter-Drinfeld modules in
$_{H}\mathcal{YD}^{H}$, $\mathcal{YD}_{H}^{H}$, and $^{H}\mathcal{YD}_{H}$.
\end{rem}

\subsection{Dimensions of $N_{W}$ and Applications }

Recall that an $H$-(co)module algebra (resp. $H$-module coalgebra) is called
$H$\textit{-simple} if it has no non-trivial $H$-(co)stable ideal (resp.
$H$-stable subcoalgebra).

Assume that $(H,R)$ is a semisimple and cosemisimple quasi-triangular Hopf
algebra over $k.$ Let $D$ be a minimal $H$-adjoint-stable subcoalgebra of
$H_{R}$, and let $W$ be a finite dimensional nonzero left $D$-comodule. We
will prove that the left $H$-comodule algebra $N_{W}$ is $H$-simple. By
Theorem \ref{thmDHM=NWM}, the functor $F=W^{\ast}\square_{D}\bullet:{}_{H}%
^{D}\mathcal{M}\rightarrow{}\mathcal{M}_{N_{W}}$ is a category equivalence.
Note that $_{H}^{D}\mathcal{M}$ is an indecomposable left module category over
$_{H^{cop}}\mathcal{M}$. Then there is a left $_{H^{cop}}\mathcal{M}$-module
category structure on $\mathcal{M}_{N_{W}}$, such that $F$ is a module
functor. Specifically, the module action $\otimes:{}_{H^{cop}}\mathcal{M}%
\times\mathcal{M}_{N_{W}}\rightarrow\mathcal{M}_{N_{W}}$ is given by $X\otimes
U=X\otimes_{k}U$ for $X\in{}_{H^{cop}}\mathcal{M}$ and $U\in{}\mathcal{M}%
_{N_{W}}$ with right $N_{W}$-action
\[
\left(  x\otimes u\right)  a=\sum\bar{S}\left(  a_{\left\langle
-1\right\rangle }\right)  x\otimes ua_{\left\langle 0\right\rangle }%
,\quad\forall a\in N_{W},x\in X,u\in U.
\]
Since $_{H}^{D}\mathcal{M}$ is an indecomposable $_{H^{cop}}\mathcal{M}%
$-module category, $\mathcal{M}_{N_{W}}$ is also indecomposable. Thus we have

\begin{prop}
Let $D$ be a minimal $H$-adjoint-stable subcoalgebra of $H_{R}$, and let $W$
be a finite dimensional nonzero left $D$-comodule. Then $N_{W}$ is $H$-simple
as a left $H$-comodule algebra.
\end{prop}

The following result due to Skryabin will be used to get a formula on the
dimension of $N_{W}$.

\begin{lem}
[{{\cite[Theorem 4.2]{skryabin2007projectivity}}}]\label{lemskryabin}Let $A$
be a finite dimensional semisimple right $H$-simple comodule algebra. Let
$0\neq M\in\mathcal{M}_{A}^{H}$ (resp. $_{A}\mathcal{M}^{H}$), $\dim
M<\infty.$ Then $M^{t}$ is free as a right (resp. left) $A$-module for some
$t\in\mathbb{N}^{+}.$

\begin{proof}
This result follows from the proof of \cite[Theorem 4.2]%
{skryabin2007projectivity}.
\end{proof}
\end{lem}

\begin{prop}
\label{propdimUctW}Let $D$ be a minimal $H$-adjoint-stable subcoalgebra of
$H_{R}$, and let $W$ be a finite dimensional nonzero left $D$-comodule. If
$V\in{}_{H}^{D}\mathcal{M}$ and $\dim V<\infty$, then
\[
\dim\left(  W^{\ast}\square_{D}V\right)  \dim D=\dim V\dim W.
\]
In particular, for finite dimensional left $D$-comodules $W,W^{\prime}$, we
have
\[
\dim N_{WW^{\prime}}\dim D=\dim H\dim W\dim W^{\prime}.
\]

\end{prop}

\begin{proof}
By assumption, $D$ is a left $H$-simple module coalgebra with $H$-action
$\cdot_{ad}$. Using Lemma \ref{lemskryabin}, we have some $\left(  s,t\right)
\in\mathbb{N}^{+}\times\mathbb{N}^{+},$ such that $V^{t}\cong\left(  D^{\ast
}\right)  ^{s}$ as right $D^{\ast}$-modules. So $V^{t}\cong D^{s}$ as left
$D$-comodules, then
\[
t\dim V=s\dim D.
\]
On the other hand,
\[
\left(  W^{\ast}\square_{D}V\right)  ^{t}\cong W^{\ast}\square_{D}V^{t}\cong
W^{\ast}\square_{D}D^{s}\cong W^{\ast s}.
\]
Therefore,%
\[
\dim\left(  W^{\ast}\square_{D}V\right)  =\dfrac{s}{t}\dim W=\dfrac{\dim V\dim
W}{\dim D}.
\]
Particularly, let $V=H\otimes W^{\prime}$, then we get
\[
\dim N_{WW^{\prime}}\dim D=\dim H\dim W\dim W^{\prime}.
\]

\end{proof}

\begin{prop}
Assume that $k$ is algebraically closed of characteristic zero. Let $D$ be a
minimal $H$-adjoint-stable subcoalgebra of $H_{R}$. If $W$ is a finite
dimensional nonzero left $D$-comodule, then for any irreducible right $N_{W}$-module
$U$,
\[
(\dim U\dim W)\mid\dim N_{W}.
\]
Especially, $N_{W}$ is an algebra over which the dimension of each irreducible
right $N_{W}$-module divides the dimension of $N_{W}$.
\end{prop}

\begin{proof}
Let $V=U\otimes_{N_{W}}\left(  H\otimes W\right)  .$ According to Theorem
\ref{thmDHM=NWM}, $V$ is an irreducible object of ${}_{H}^{H}\mathcal{YD}$. By
the result of Etingof and Gelaki \cite[Theorem 1.4]{Etingof1997Some}, $\dim
V\mid\dim H.$ On the other hand, observe that $H\otimes W\in{}_{N_{W}}%
^{H}\mathcal{M}$ with $H$-coaction given by $\rho\left(  h\otimes w\right)
=\sum S\left(  h_{\left(  2\right)  }\right)  \otimes h_{\left(  1\right)
}\otimes w$, where $h\in H,$ $w\in W$. Since $N_{W}$ is $H$-simple, by Lemma
\ref{lemskryabin}, there exists some $\left(  s,t\right)  \in\mathbb{N}%
^{+}\times\mathbb{N}^{+}$, such that $\left(  H\otimes W\right)  ^{t}%
\cong\left(  N_{W}\right)  ^{s}$ as left $N_{W}$-modules. On one hand,
\[
\dfrac{s}{t}=\dfrac{\dim H\dim W}{\dim N_{W}}.
\]
On the other hand, as a linear space
\begin{align*}
V^{t}  &  \cong\left(  U\otimes_{N_{W}}\left(  H\otimes W\right)  \right)
^{t}\cong U\otimes_{N_{W}}\left(  H\otimes W\right)  ^{t}\\
&  \cong U\otimes_{N_{W}}\left(  N_{W}\right)  ^{s}\\
&  \cong U^{s}.
\end{align*}
Then we obtain $\dfrac{\dim H\dim W}{\dim N_{W}}=\dfrac{\dim V}{\dim U}$,
whence $\dfrac{\dim N_{W}}{\dim U\dim W}=\dfrac{\dim H}{\dim V}$ is an integer.
\end{proof}

\begin{exam}
Let $k=\mathbb{C}$, and $H=H_{8}$ be the Kac-Paljutkin algebra, a semisimple
Hopf algebra constructed in \cite{masuoka1995semisimple}. It is generated as
algebra by $x,y,z$ with relations
\[
x^{2}=y^{2}=1,\text{ }z^{2}=\dfrac{1}{2}\left(  1+x+y-xy\right)  ,\text{
}xy=yx,\text{ }zx=yz,\text{ }zy=xz.
\]
Its coalgebra structure and antipode is determined by%
\begin{align*}
\Delta\left(  x\right)   &  =x\otimes x,\text{ }\Delta\left(  y\right)
=y\otimes y,\text{ }\\
\Delta\left(  z\right)   &  =\dfrac{1}{2}\left(  1\otimes1+1\otimes
x+y\otimes1-y\otimes x\right)  \left(  z\otimes z\right)  ,\\
\varepsilon\left(  x\right)   &  =\varepsilon\left(  y\right)  =\varepsilon
\left(  z\right)  =1,\text{ }S\left(  x\right)  =x,\text{ }S\left(  y\right)
=y,\text{ }S\left(  z\right)  =z.
\end{align*}
We found that $H$ is quasi-triangular with R-matrix
\[
R=\dfrac{1}{2}\left(  1\otimes1+1\otimes x+y\otimes1-y\otimes x\right)  .
\]
As a coalgebra, $H_{R}$ is pointed with grouplike elements $G\left(
H_{R}\right)  =\left\{  1,x,y,xy,g_{1},g_{2},g_{3},g_{4}\right\}  ,$ where
\begin{align*}
g_{1}  &  =\dfrac{1+i}{2}z+\dfrac{1-i}{2}yz,\text{ }g_{2}=\dfrac{1-i}%
{2}xz+\dfrac{1+i}{2}xyz,\text{ }\\
g_{3}  &  =\dfrac{1-i}{2}z+\dfrac{1+i}{2}yz,\text{ }g_{4}=\dfrac{1+i}%
{2}xz+\dfrac{1-i}{2}xyz.\text{ }%
\end{align*}
Moreover, we have $D_{1}=k1,$ $D_{2}=kxy,$ $D_{3}=kx\oplus ky,$ $D_{4}%
=kg_{1}\oplus kg_{2},$ $D_{5}=kg_{3}\oplus kg_{4}$ are all the minimal
$H$-adjoint-stable subcoalgebras of $H_{R}$. For each grouplike element, we
get the $R$-adjoint-stable algebras:
\begin{align*}
N_{k1}  &  =N_{kxy}=H,\\
N_{kx}  &  =N_{ky}=\mathrm{span}\left\{  1,x,y,xy\right\}  ,\\
N_{kg_{1}}  &  =N_{kg_{2}}=\mathrm{span}\left\{  1,xy,z+ixz,iyz+xyz\right\}
,\\
N_{kg_{3}}  &  =N_{kg_{4}}=\mathrm{span}\left\{  1,xy,iz+xz,yz+ixyz\right\}  .
\end{align*}
Because $N_{kg_{j}}$ $\left(  1\leq j\leq4\right)  $ are 4-dimensional
semisimple algebras and all contain at least two central elements $1,xy,$
$N_{kg_{j}}$ are commutative. Then using Proposition \ref{PropifC=kg} all the
irreducible Yetter-Drinfeld modules over $H$ can be obtained. Moreover, the
number and the dimensions of irreducible Yetter-Drinfeld modules corresponding
to each $D_{j}$ ($1\leq j\leq5$) are given by the following table:
\[%
\begin{tabular}
[c]{l|l|c|c}%
\begin{tabular}
[c]{l}%
simple YD sub-\\
module $D$ of $H$%
\end{tabular}
&
\begin{tabular}
[c]{l}%
$N_{W}$ where $W$ is a sim-\\
ple left coideal of $D$%
\end{tabular}
& \#$\{V\in\operatorname{Irr}\left(  {}_{H}^{D}\mathcal{M}\right)  \}$ &
\begin{tabular}
[c]{l}%
the list of $\dim V$,\\
where $V\in\operatorname{Irr}\left(  {}_{H}^{D}\mathcal{M}\right)  $%
\end{tabular}
\\\hline\hline
\multicolumn{1}{c|}{$D_{1}$} & \multicolumn{1}{|c|}{$H$} & $5$ & $1,1,1,1,2$\\
\multicolumn{1}{c|}{$D_{2}$} & \multicolumn{1}{|c|}{$H$} & $5$ & $1,1,1,1,2$\\
\multicolumn{1}{c|}{$D_{3}$} & \multicolumn{1}{|c|}{$4$-dim comm. algebra} &
$4$ & $2,2,2,2$\\
\multicolumn{1}{c|}{$D_{4}$} & \multicolumn{1}{|c|}{$4$-dim comm. algebra} &
$4$ & $2,2,2,2$\\
\multicolumn{1}{c|}{$D_{5}$} & \multicolumn{1}{|c|}{$4$-dim comm. algebra} &
$4$ & $2,2,2,2$%
\end{tabular}
\ \
\]

Therefore, there are 22 non-isomorphic irreducible Yetter-Drinfeld modules
over $H_{8}$. The irreducible Yetter-Drinfeld modules over $H_{8}$ were also
listed in \cite{Hu2007,Shi2016Finite}.

Clearly, $ZG(H)=\left\{  1,xy\right\}  $, and there are four 1-dimensional
$H$-modules. Therefore by Corollary~\ref{cor_1-dim} there are eight
1-dimensional Yetter-Drinfeld modules in ${}_{H}^{H}\mathcal{YD}$, which are
counted in the table.
\end{exam}

\noindent\textbf{{\large Acknowledgement.}} The authors would like to thank
the referee for his/her technical advice and helpful comments.



\begin{thebibliography}{99}                                                                                               %


\bibitem {Andruskiewitsch1998Braided}N.~Andruskiewitsch and M.~Gra\~{n}a.
\newblock Braided {H}opf algebras over non-abelian finite groups.
\newblock {\em Bol. Acad. Nac. Cienc. (C\'{o}rdoba)}, 63:45--78, 1999.
\newblock Colloquium on Operator Algebras and Quantum Groups (Spanish)
(Vaquer\'{\i}as, 1997).

\bibitem {andruskiewitsch2007module}N.~Andruskiewitsch and J.~M. Mombelli.
\newblock On module categories over finite-dimensional {H}opf algebras.
\newblock {\em J. Algebra}, 314(1):383--418, 2007.

\bibitem {cohen2010higman}M.~Cohen and S.~Westreich. \newblock Higman ideals
and {V}erlinde-type formulas for {H}opf algebras. \newblock In \emph{Ring and
module theory}, Trends Math., pages 91--114. Birkh\"{a}user/Springer Basel AG,
Basel, 2010.

\bibitem {cohen2010structure}M.~Cohen and S.~Westreich. \newblock Structure
constants related to symmetric {H}opf algebras. \newblock {\em J. Algebra},
324(11):3219--3240, 2010.

\bibitem {cohen2011conjugacy}M.~Cohen and S.~Westreich. \newblock Conjugacy
classes, class sums and character tables for {H}opf algebras.
\newblock {\em Comm. Algebra}, 39(12):4618--4633, 2011.

\bibitem {dijkgraaf1992quasi}R.~Dijkgraaf, V.~Pasquier, and P.~Roche.
\newblock Quasi {H}opf algebras, group cohomology and orbifold models.
\newblock {\em Nuclear Phys. B Proc. Suppl.}, 18B:60--72 (1991), 1990.
\newblock Recent advances in field theory (Annecy-le-Vieux, 1990).

\bibitem {drinfeld1986quantum}V.~G. Drinfeld. \newblock Quantum groups.
\newblock In \emph{Proceedings of the {I}nternational {C}ongress of
{M}athematicians, {V}ol. 1, 2 ({B}erkeley, {C}alif., 1986)}, pages 798--820.
Amer. Math. Soc., Providence, RI, 1987.

\bibitem {drinfeld1990almost}V.~G. Drinfeld. \newblock Almost cocommutative
{H}opf algebras. \newblock {\em Algebra i Analiz}, 1(2):30--46, 1989.

\bibitem {Etingof1997Some}P.~Etingof and S.~Gelaki. \newblock Some properties
of finite-dimensional semisimple {H}opf algebras.
\newblock {\em Math. Res. Lett.}, 5(1-2):191--197, 1998.

\bibitem {Etingof2015tensor}P.~Etingof, S.~Gelaki, D.~Nikshych, and V.~Ostrik.
\newblock {\em Tensor categories}, volume 205 of \emph{Mathematical Surveys
and Monographs}. \newblock American Mathematical Society, Providence, RI, 2015.

\bibitem {gould1993quantum}M.~D. Gould. \newblock Quantum double finite group
algebras and their representations. \newblock {\em Bull. Austral. Math. Soc.},
48(2):275--301, 1993.

\bibitem {Hu2007}J.~Hu and Y.-H. Zhang. \newblock The {$\beta$}-character
algebra and a commuting pair in {H}opf algebras.
\newblock {\em Algebr. Represent. Theory}, 10(5):497--516, 2007.

\bibitem {lyubashenko1986superanalysis}V.~V. Lyubashenko.
\newblock {\em Superanalysis and solutions to the triangles equation}.
\newblock PhD thesis, Phy.-Mat. Sciences, Kiev,(in Russian), 1986.

\bibitem {Majid1991Braided}S.~Majid. \newblock Braided groups and algebraic
quantum field theories. \newblock {\em Lett. Math. Phys.}, 22(3):167--175, 1991.

\bibitem {majid1991doubles}S.~Majid. \newblock Doubles of quasitriangular
{H}opf algebras. \newblock {\em Comm. Algebra}, 19(11):3061--3073, 1991.

\bibitem {masuoka1995semisimple}A.~Masuoka. \newblock Semisimple {H}opf
algebras of dimension {$6,8$}. \newblock {\em Israel J. Math.},
92(1-3):361--373, 1995.

\bibitem {Montgomery1993Hopf}S.~Montgomery.
\newblock {\em Hopf algebras and their actions on rings}, volume~82 of \emph{
CBMS Regional Conference Series in Mathematics}. \newblock Published for the
Conference Board of the Mathematical Sciences, Washington, DC; by the American
Mathematical Society, Providence, RI, 1993.

\bibitem {Ostrik2003module}V.~Ostrik. \newblock Module categories, weak {H}opf
algebras and modular invariants. \newblock {\em Transform. Groups},
8(2):177--206, 2003.

\bibitem {RADFORD1994583}D.~E. Radford. \newblock The trace function and
{H}opf algebras. \newblock {\em J. Algebra}, 163(3):583--622, 1994.

\bibitem {reshetikhin1988quantum}N.~Y. Reshetikhin and M.~A.
Semenov-Tian-Shansky. \newblock Quantum {$R$}-matrices and factorization
problems. \newblock {\em J. Geom. Phys.}, 5(4):533--550 (1989), 1988.

\bibitem {schneider2001some}H.-J. Schneider. \newblock Some properties of
factorizable {H}opf algebras. \newblock {\em Proc. Amer. Math. Soc.},
129(7):1891--1898, 2001.

\bibitem {Shi2016Finite}Y.-X. Shi. \newblock Finite dimensional {Nichols}
algebras over {Kac-Paljutkin} algebra {$H_{8}$}.
\newblock {\em math.QA:1612.03262v4}, 2016.

\bibitem {skryabin2007projectivity}S.~Skryabin. \newblock Projectivity and
freeness over comodule algebras. \newblock {\em Trans. Amer. Math. Soc.},
359(6):2597--2623, 2007.

\bibitem {MR0252485}M.~E. Sweedler. \newblock {\em {Hopf} algebras}.
\newblock Mathematics Lecture Note Series. W. A. Benjamin, Inc., New York, 1969.

\bibitem {yetter1990quantum}D.~N. Yetter. \newblock Quantum groups and
representations of monoidal categories.
\newblock {\em Math. Proc. Cambridge Philos. Soc.}, 108(2):261--290, 1990.

\bibitem {Zhu2015Relative}H.-X. Zhu. \newblock Relative {Y}etter-{D}rinfeld
modules and comodules over braided groups. \newblock {\em J. Math. Phys.},
56(4):041706, 11, 2015.
\end{thebibliography}
\end{document}